\documentclass[10pt]{article}
\usepackage{amsmath}
\usepackage{amsthm}
\usepackage{amsfonts}
\usepackage{amssymb}
\usepackage{latexsym} 
\usepackage{xcolor}

\usepackage[matrix,tips,graph,curve]{xy}



\makeatletter 
\@addtoreset{equation}{section}
\makeatother

\theoremstyle{plain}
\newtheorem{theorem}[equation]{Theorem}
\newtheorem{corollary}[equation]{Corollary}
\newtheorem{lemma}[equation]{Lemma}
\newtheorem{proposition}[equation]{Proposition}

\theoremstyle{definition}
\newtheorem{define}[equation]{Definition}

\newtheorem{remark}[equation]{Remark}

\newcommand{\IC}{\mathbb{C}}

\newcommand{\IQ}{\mathbb{Q}}
\newcommand{\IR}{\mathbb{R}}

\newcommand{\End}{\mathrm{End}}

\renewcommand\dim{{\rm dim\,}}
\renewcommand\det{{\rm det\,}}

\newcommand\vol{\mathrm{vol}}

\def\d/{/\mspace{-6.0mu}/}

\newcommand{\p}{\partial}

\begin{document}

\title{Geometry of the Smallest 1-form Laplacian Eigenvalue on Hyperbolic Manifolds}
\author{Michael Lipnowski and Mark Stern}
\footnotetext{Duke University and University of Toronto Departments of Mathematics;  \\
e-mail:  stern@math.duke.edu, partially supported by by NSF grant DMS 1005761 \\
e-mail: malipnow@math.utoronto.ca}
\date{}

 \maketitle

\section{Introduction} \label{intro}
Let $M^n$ be a closed Riemannian manifold.  The geometry of the smallest positive eigenvalue $\lambda_1^0(M)$ of the 0-form Laplace operator is well studied.  Work of Cheeger \cite{Cheeger} and Buser \cite{Buser} proves that $\lambda_1^0(M)$ is comparable to the square of the Cheeger isoperimetric constant of $M.$  

Much less is known about the smallest positive eigenvalue $\lambda_1^q(M)$ of the $q$-form Laplace operator for $0<q<n$.  In this paper, motivated by questions arising in the study of torsion cohomology of closed arithmetic \emph{hyperbolic} manifolds $M,$ we prove geometric upper and lower bounds for $\lambda_1^1(M).$

\subsection{Main results}
Let $M$ be a closed hyperbolic $n$-manifold.  Its fundamental group $\pi_1(M)$ acts by isometries on $\mathbb{H}^n.$  For $\gamma \in \pi_1(M),$ let $\ell(\gamma)$ denote the translation length of $\gamma.$  

Fix a basepoint $q_0 \in M.$  For $x,y \in \mathbb{H}^n,$ let $\alpha_{x,y}$ denote the oriented geodesic segment from $x$ to $y.$  If $\gamma$ bounds, define the \emph{area of $\gamma$} to be
\begin{equation}\label{area}\mathrm{Area}(\gamma) := \inf_{\partial S = \alpha_{q_0, \gamma q_0}} \mathrm{area}(S).\end{equation}

Define the \emph{stable area of $\gamma$}, denoted $\mathrm{sArea}(\gamma),$ to be 
\begin{equation}\label{sarea}\mathrm{sArea}(\gamma) := \inf \left\{\frac{\mathrm{area}(\gamma^m)}{m}:\gamma^m\text{ bounds} \right\},\end{equation}
assuming that $\gamma^k$ bounds for some integer $k.$  The quantity $\mathrm{sArea}(\gamma)$ is independent of the basepoint $q_0.$  

Under the latter assumption, recall that the \emph{stable commutator length of} $\gamma,$ denoted $\mathrm{scl}(\gamma),$ is defined by  
\begin{equation*}
\mathrm{scl}(\gamma) = \inf_m\inf_{\text{connected }S:\p S = m\gamma}\frac{\max\{-\chi(S),0\}}{m}.
\end{equation*}

On hyperbolic manifolds, stable area is always bounded above by $4\pi$ times stable commutator length \cite{C1}.

Let $\lambda_1^q(M)_{d^\ast}$ (resp. $\lambda_1^q(M)_d$) denote the smallest positive eigenvalue of the $q$-form Laplacian acting on $d^\ast \Omega^{q+1}(M)$ (resp. $d \Omega^{q-1}(M)$).  Then
$$\lambda_1^q(M) = \min \{ \lambda_1^q(M)_d, \lambda_1^q(M)_{d^\ast} \},$$
by the Hodge decomposition.

\begin{theorem}[Geometric Upper Bound for $\frac{1}{\lambda_1^1(M)_{d^\ast}}$] \label{introupperbound}
Let $M_0$ be a closed hyperbolic $n$-manifold.  Let $M$ be an arbitrary finite cover of $M_0$ with first betti number 0.  Then


\begin{align*}
\frac{1}{\sqrt{\lambda_1^1(M)_{d^\ast}}} &\leq   C^2 \left( 2\pi V+  D  \sup_{\gamma \in \pi_1(M): \ell(\gamma) \leq D}  \frac{\mathrm{sArea}(\gamma)}{\ell(\gamma)} \right)  + \frac{C}{2}   \sqrt{\vol(M)} \nonumber.\\
\end{align*}

The quantity $C$ is defined in Proposition \ref{supnormeigenfunction}; it is uniformly bounded above when the injectivity radius of $M$ is bounded below and $\lambda_1^1(M)$ is bounded above.

The quantities $V$ and $D$ are defined in Theorem \ref{bodyupperboundpositivebettinumber}; they respectively satisfy
$$V \leq V_{M_0} \cdot \vol(M) \text{ and } D \leq D_{M_0} \cdot \mathrm{diam}(M)$$
for constants $V_{M_0}, D_{M_0}$ depending only on $M_0$ (in an explicit manner to be described later).
\end{theorem}

Let $\gamma\to [\gamma]$ denote the quotient map from $\pi_1(M)\to H_1(M,\mathbb{Q}).$
 We can extend the bounds of Theorem \ref{introupperbound} to the case $n = 3, b_1(M) = 1$:

\begin{theorem}[``Regulator-independent" Geometric Upper Bound for $\frac{1}{\lambda_1^1(M)_{d^\ast}}$ when $n = 3, b_1(M) = 1$] \label{introupperboundbettinumber1}
Let $M_0$ be a closed hyperbolic 3-manifold.  Let $M$ be an arbitrary finite cover of $M_0$ with $b_1(M) = 1.$  Fix $\delta > 0.$  Then 

\begin{align*}
 \frac{(1 - E)}{\sqrt{\lambda_1^1(M)_{d^\ast}}}  &\leq  C^2   \left( 3\pi V+2\sqrt{2} D^2 \cdot \vol(M)^{\delta + 1/2}  \sup_{\gamma \in \pi_1(M): [\gamma] = 0} \frac{\mathrm{sArea}(\gamma)}{\ell( \gamma )} + 5\pi  \right)  + \frac{C}{2}  \sqrt{\vol(M)}.  \nonumber \\
\end{align*}

The quantities $V, C, D$ are as in Theorem \ref{introupperbound}.  The quantity $E$ satisfies 
$$0\leq E \leq C \cdot \vol(M)^{-\delta}.$$
\end{theorem}

Notably, the upper bound in Theorem \ref{introupperboundbettinumber1} does not depend on $A',$ the maximum over a generating set $\{ [S] \}$ of $H_2(M,\mathbb{Z})$ of the least area representative of $[S].$  

When we don't impose restrictions on  $b_1(M),$ Theorem \ref{bodyupperboundpositivebettinumber} proves an upper bound for $\frac{1}{\lambda_1^1(M)_{d^\ast}}$ in terms of $\sup_{\gamma \in \pi_1(M): [\gamma] = 0} \frac{\mathrm{sArea}(\gamma)}{ \ell(\gamma)},$ which depends a priori on $A'.$  
 \bigskip

Let $K_0$ be a triangulation of $M_0.$  If $M$ is a finite cover, let $K$ denote the pullback triangulation of $K_0.$  The cochain complex $C^\bullet(M;K)$ maps into $\Omega^\bullet(M)$ by the Whitney map \cite{RS}.  Endow $C^\bullet(M;K)$  with the norm induced from the $L^2$-norm on $\Omega^\bullet(M)$ via the Whitney map.  Let $\lambda_{1,\mathrm{Whitney}}^q(M)$ denote the smallest positive eigenvalue of the associated Whitney Laplacian on $C^q(M;K).$  Let $\lambda_{1, \mathrm{Whitney}}^q(M)_{d^\ast}$ (resp. $\lambda_{1,\mathrm{Whitney}}^q(M)_d$) denote the smallest positive eigenvalue of the Whitney Laplacian acting on $d^{\ast}_{\mathrm{Whitney}} C^{q+1}(M;K)$ (resp. $d_{\mathrm{Whitney}} C^{q-1}(M;K)$).  Then
$$\lambda_{1,\mathrm{Whitney}}^q(M) = \min \{ \lambda_{1,\mathrm{Whitney}}^q(M)_d,\lambda_{1,\mathrm{Whitney}}^q(M)_{d^\ast}  \}$$
by the Hodge decomposition.

\begin{theorem}[Geometric Lower Bound for $\frac{1}{\lambda_{1,\mathrm{Whitney}}^1(M)_{d^\ast}}$] \label{introlowerbound}
Let $M_0$ be a closed hyperbolic $n$-manifold.  Let $M$ be an arbitrary finite cover of $M_0.$
If some multiple of $\gamma \in \pi_1(M)$ bounds, then
$$\left( \frac{\mathrm{scl}(\gamma)}{\ell(\gamma)} \right)^2 \leq W_{M_0} \cdot \frac{\vol(M) \cdot \mathrm{diam}(M)^2}{\lambda_{1, \mathrm{Whitney}}^1(M)_{d^\ast}},$$
for some constant $W_{M_0}$ depending only on $M_0$ (described in Theorem \ref{whitneylambda1controlsscltriangulationindependent}).
\end{theorem}

\begin{theorem}[Comparison between $\lambda_{1, \mathrm{Whitney}}^1(M)_{d^\ast}$ and $\lambda_1^1(M)_{d^\ast}$] \label{introwhitneyderhamcomparison}
Let $M_0$ be a closed hyperbolic $n$-manifold.  Let $M$ be an arbitrary finite cover of $M_0.$  Then either
$$\lambda_{1,\mathrm{Whitney}}^1(M)_{d^\ast} \geq \frac{1}{4 G_{M_0}^2 C_{M_0}^2} \vol(M)^{-1}$$
or
$$\lambda_1^1(M)_{d^\ast} \leq 4G_{M_0}^2 \vol(M)\lambda_{1,\mathrm{Whitney}}^1(M)_{d^\ast}.$$
The constants $C_{M_0},G_{M_0}$ depend only on $M_0$ and are defined in Propositions \ref{comparisonswithwhitneysupnorm} and \ref{whitneyderhamcomparison} respectively.
\end{theorem}

The combinatorial-to-Riemannian comparison in Theorem \ref{introwhitneyderhamcomparison} shows that one of the following two alternatives must hold:
\begin{itemize}
\item[(1)]
$\lambda_{1,\mathrm{Whitney}}^1(M)_{d^\ast} \geq \frac{1}{4G_{M_0}^2 C_{M_0}^2} \vol(M)^{-1}.$  In this case, Theorem \ref{introlowerbound} shows that every $\gamma$ which is trivial in $H_1(M,\mathbb{Q})$ satisfies 
\begin{equation} \label{ifwhitneylambda1islarge}
\frac{\mathrm{scl}(\gamma)}{\ell(\gamma)} \leq 2 G_{M_0} C_{M_0} \cdot \sqrt{W_{M_0}} \vol(M) \cdot \mathrm{diam}(M).
\end{equation}

Inserting \eqref{ifwhitneylambda1islarge} into Theorems \ref{introupperbound} (and \ref{introupperboundbettinumber1}) implies that 
for every $\delta > 0,$ there is an upper bound


\begin{equation} \label{reallygoodoutcome}
\frac{1}{\sqrt{\lambda_1^1(M)_{d^\ast}}} \leq \begin{cases}
E_{M_0} \cdot \vol(M) \cdot \mathrm{diam}(M)^2 &\text{ if } b_1(M) = 0 \\
E_{M_0,\delta} \cdot \vol(M)^{3/2 + \delta} \cdot \mathrm{diam}(M)^3 & \text{ if } n=3 \text{ and } b_1(M) = 1,
\end{cases}
\end{equation}

where $E_{M_0}, E_{M_0, \delta}$ depend only on $M_0, \delta.$  The inequality \eqref{reallygoodoutcome} yields applications to the growth of $H_1(M,\mathbb{Z})_{\mathrm{tors}}$ for towers of hyperbolic 3-manifolds; see \S \ref{motivation} for details.

\item[(2)]  
Suppose (1) does not hold.  Then Theorems \ref{introlowerbound} and \ref{introwhitneyderhamcomparison} imply that  
\begin{align} \label{lowerboundwhitneyupperboundderham}
\left( \frac{\mathrm{scl}(\gamma)}{\ell(\gamma)} \right)^2 & \leq 4 G_{M_0}^2 \cdot W_{M_0} \cdot \frac{\vol(M)^2 \cdot \mathrm{diam}(M)^2}{\lambda_1^1(M)_{d^\ast}}.\end{align}
 Theorem \ref{introupperbound} yields an upper bound for the right hand side of \eqref{lowerboundwhitneyupperboundderham} which is quadratic in $\frac{\mathrm{sArea}(\gamma)}{\ell(\gamma)}$. 
Thus we have upper and lower bounds for $\frac{1}{\lambda_1^1(M)_{d^\ast}}$ which have the same order of magnitude, up to terms of polynomial size in $\vol(M).$  This implies: 

\begin{quote}
\emph{$\frac{1}{\lambda_1^1(M)_{d^\ast}}$ is at most polynomial in $\vol(M)$ if and only if every $\gamma \in \pi_1(M)$ with $[\gamma]=0$ has stable area at most $\ell(\gamma) \cdot (\text{polynomial in }\vol(M) ).$ }
\end{quote}   
\end{itemize}

Our main results comparing $\lambda_{1,\mathrm{Whitney}}^1(M)_{d^\ast}, \lambda_1^1(M)_{d^\ast},$ and stable commutator length also prove new relationships between the 1-form spectra of closed hyperbolic manifolds $M \subset N,$ where $M$ is geodesically embedded in $N.$

\begin{theorem} \label{surfacetheft}  
Let $N_0$ be a closed hyperbolic $n$-manifold, $n > 3.$  Let $M_0 \subset N_0$ be a totally geodesic submanifold.  Suppose $N \xrightarrow{\pi} N_0$ is an arbitrary finite cover.
Let $M =$ (a connected component of) $\pi^{-1}(M_0).$  Suppose that there is a covering $p: N' \rightarrow N$ of degree $d$ for which
\begin{itemize}
\item
the submanifold $M$ lifts to $N'$

\item
$N'$ retracts onto $M.$
\end{itemize}
Then 
$$\frac{1}{\lambda_1^1(M)_{d^\ast}} \leq F_{M_0} \cdot \mathrm{diam}(M)^2 \cdot \left( d^2 \cdot \vol(N) \cdot \mathrm{diam}(N) \right)^2 \cdot \frac{1}{\lambda_1^1(N')_{d^\ast}}$$
for a constant $F_{M_0}$ depending only on $M_0$ (in a manner to be described explicitly later).
\end{theorem}

A rich family of (arithmetic) examples satisfying the hypotheses of Theorem \ref{surfacetheft} is provided by the work of Bergeron, Haglund, and Wise \cite{BHW}.  The 1-form spectra of hyperbolic $n$-manifolds, $n > 3,$ are typically much easier to bound away from 0 than the 1-form spectrum of hyperbolic $3$-manifolds because $\lambda_1^1(\mathbb{H}^n) > 0$ if $n > 3$ while $\lambda_1^1(\mathbb{H}^3) = 0$; see \S \ref{1formshighdimension} for further discussion.  Theorem \ref{surfacetheft} may therefore be useful for proving good lower bounds for the 1-form spectra of closed hyperbolic 3-manifolds.  \medskip

This work began as an attempt to prove that $\frac{1}{\lambda_1^1(M)} \ll_{M_0} \vol(M)^{C}$, for some constant $C.$  This bound is motivated by applications to estimating growth of $H_1(M, \mathbb{Z})_{\mathrm{tors}},$ as we'll describe in \S \ref{motivation}.  Focusing attention on $\lambda_1^1(M)_{d^\ast}$ may appear to miss ``half the story": bounding $\lambda_1^1(M)_d = \lambda_1^0(M)$ from below.  But $\lambda_1^1(M)_d$ is much simpler to control; see \S \ref{alternatelowerboundlambda10} for further discussion.


\subsection{Motivation: Relationship to torsion cohomology growth} \label{motivation}
This paper began as an attempt to prove growth of torsion in $H_1(M,\mathbb{Z})$ for towers of closed hyperbolic 3-manifolds $M.$  

For every closed Riemannian manifold $M,$ the Cheeger-M\"{u}ller Theorem \cite{Cheeger1} \cite{Muller} relates torsion cohomology to analytic invariants of Riemannian manifolds:

\begin{equation} \label{cheegermuller}
\sum_{q = 0}^{\dim M} (-1)^q \left( \log |H^q(M,\mathbb{Z})_{\mathrm{tors}}| + \log( R^{\dim M - q}(M)) + \frac{q}{2} \sum_{\lambda \in \text{spectrum}(\Delta_q)\setminus\{0\}}^{\mathrm{reg}}\log \left( \frac{1}{\lambda}  \right) \right) = 0.
\end{equation}

The regulator $R^q(M)$ measures the volume of $H^q(M,\mathbb{Z}) \backslash H^q(M,\mathbb{R})$ with $L^2$-metric induced from harmonic forms.  In particular, $\log R^q(M) = 0$ if $H^q(M,\mathbb{R}) = 0.$  The notation $\sum^{\mathrm{reg}}$ means zeta-regularized sum.  

Under favorable circumstances, one hopes that for many sequences of hyperbolic 3-manifolds $M$ ``geometrically converging to $\mathbb{H}^3,$"
\begin{align*}
\frac{1}{\vol(M)} \sum_{q = 0}^{\dim M} (-1)^q \frac{q}{2} \sum_{\lambda \in \text{spectrum}(\Delta_q)\setminus\{0\}}^{\mathrm{reg}}\log \left( \frac{1}{\lambda} \right) &=: \frac{1}{\vol(M)} \cdot \log T_{\mathrm{an}}(M) \\
&\to \log T_{\mathrm{an}}^{(2)}(\mathbb{H}^3) = -\frac{1}{6\pi},
\end{align*}
where $T_{\mathrm{an}}^{(2)}$ is the $L^2$-analytic torsion of $\mathbb{H}^3$ \cite[\S 3]{Luck}.  Nonetheless, convergence of analytic torsion to its expected $L^2$-limit has not been proven for even a single sequence of closed hyperbolic 3-manifolds converging geometrically to $\mathbb{H}^3.$  

Equation \eqref{cheegermuller} shows that small $q$-form Laplacian eigenvalues and large $R^{\dim M - q}(M)$\footnote{Large $R^{\dim M - q}(M)$ is corresponds to homology classes in $H_q(M,\mathbb{Z})$ whose minimal complexity is very large.  See \cite{BSV} for further details.} suppress torsion cohomology in degree $q.$  

Bergeron and Venkatesh \cite{BV} found many interesting examples of non-trivial unimodular metrized local systems $L$ of free abelian groups for which $R^{\dim M - q}(M; L) = 1$ and without small eigenvalues\footnote{Bergeron and Venkatesh construct \emph{strongly acyclic metrized local systems}, for which the $q$-form Laplace operators for a tower of hyperbolic 3-manifolds admit a uniform spectral gap for every $q.$  See \cite[\S 4]{BV} for details and \cite[\S 8]{BV} for constructions.}.  In the absence of these two torsion suppressors, Bergeron and Venkatesh prove a limit multiplicity formula showing that the analytic torsion $\frac{\log T_{\mathrm{an}}(M, L)}{\vol(M)}$ approaches its expected $L^2$-limit.  Upon applying M\"{u}ller's generalization of the Cheeger-M\"{u}ller theorem to metrized unimodular local systems $L$ \cite{Muller1}, this proves growth of torsion in the cohomology $H^{\ast}(M, L).$

For the trivial local system $\mathbb{Z}$ and many others, the analytic obstructions of small eigenvalues and complicated cycles alluded to above are genuine and present interesting geometric problems.
 
Bergeron, Venkatesh, and Seng\"{u}n \cite[Theorem 1.2]{BSV} have codified the obstruction to proving the torsion cohomology growth theorems via the methods of \cite{BV}:   

\begin{theorem}[\cite{BSV}, Theorem 1.2] \label{sufficientconditionstorsiongrowth}
Let $M_0$ be a closed hyperbolic 3-manifold and $M_n \rightarrow M_0$ normal coverings for which $\bigcap \pi_1(M_n) = \{ 1\}.$  Suppose that 
\begin{itemize}
\item
$M_n$ has ``small betti numbers", i.e.
\begin{equation} \label{smallbettinumbers}
b_1(M_n) = o \left( \frac{\vol(M_n)}{\log \vol(M_n)} \right).
\end{equation}

\item
$M_n$ has ``few small 1-form eigenvalues"\footnote{\cite[Theorem 1.2]{BSV} makes the further assumption that $M_n$ be arithmetic congruence.  Under this assumption, the 0-form Laplacian admits a uniform spectral gap and the analogous condition for ``few small 0-form eigenvalues" is automatically satisfied.     However, Lemma \ref{alternatelowerboundlambda10}  together with know bounds on almost-betti numbers in degree 0 show that ``few small 0-form eigenvalues" is satisfied even without assuming every $M_n$ is arithmetic congruence.}, i.e. 
\begin{equation} \label{fewsmalleigenvalues}
\lim_{n \to \infty} \frac{1}{\vol(M_n)} \sum_{\lambda \in \text{spectrum}(\Delta_1)\cap (0,\vol(M_n)^{-\delta})} \log \left( \frac{1}{\lambda} \right) = 0 \text{ for every } \delta > 0.
\end{equation}

\item
$M_n$ has ``simple cycles", i.e.
\begin{equation} \label{simplecycles}
H_2(M_n,\mathbb{R}) \text{ is spanned by cycles of area } \ll \vol(M_n)^C \text{ for some constant } C.
\end{equation}
\end{itemize}
Then it follows that
\begin{equation} \label{smallregulator}
\lim_{n \to \infty} \frac{\log R^q(M_n)}{\vol(M_n)} = 0
\end{equation}

\begin{equation} \label{analytictorsionlimitmultiplicity}
\lim_{n \to \infty} \frac{- \log |H_1(M_n,\mathbb{Z})_{\mathrm{tors}}|}{\vol(M_n)} = \lim_{n \to \infty} \frac{\log T_{\mathrm{an}}(M_n)}{\vol(M_n)} = -\frac{1}{6 \pi}.
\end{equation}
\end{theorem}

The main focus of \cite{BSV} was on understanding the simple cycles condition \eqref{simplecycles}.  Throughout the present paper, we focus on the small eigenvalue condition \eqref{fewsmalleigenvalues}.\footnote{We do not believe the simple cycle condition \eqref{simplecycles} and the small eigenvalue condition \eqref{fewsmalleigenvalues} should be regarded as independent.  We will return to the connection between regulators and small eigenvalues in future work.} 

Under the simplifying assumption $b_1(M_n) = 0,$ we state different set of sufficient conditions that emphasizes the connection between small eigenvalues and geometry; this is a consequence of Theorem \ref{introupperbound}.

\begin{theorem} \label{newsufficientconditionstorsiongrowth}
Let $M_0$ be a closed hyperbolic 3-manifold and $M_n \rightarrow M_0$ normal coverings for which $\bigcap \pi_1(M_n) = \{ 1\}$ and which satisfy $b_1(M_n) = 0.$  Suppose that 
\begin{itemize}
\item
$M_n$ has ``small almost-betti numbers", i.e.
\begin{equation} \label{smallalmostbettinumbers}
\sum_{\lambda \in \text{spectrum}(\Delta_1)\cap [0,\vol(M_n)^{-\delta})} 1 = o \left( \frac{\vol(M_n)}{\log \vol(M_n)} \right) \text{ for all } \delta > 0.
\end{equation}

\item
$M_n$ has ``simple almost-cycles", i.e. for some constant $D_{M_0}$ depending only on $M_0,$
\begin{equation} \label{simplealmostcycles} 
\text{If } \ell(\gamma) \leq D_{M_0} \cdot \mathrm{diam}(M_n), \text{ then } sArea(\gamma) \leq \vol(M_n)^C \text{ for some constant } C.
\end{equation}

\end{itemize}

Then it follows that

\begin{equation} \label{analytictorsionlimitmultiplicity}
\lim_{n \to \infty} \frac{- \log |H_1(M_n,\mathbb{Z})_{\mathrm{tors}}|}{\vol(M_n)} = \lim_{n \to \infty} \frac{\log T_{\mathrm{an}}(M_n)}{\vol(M_n)} = -\frac{1}{6 \pi}.
\end{equation}
\end{theorem}

The simple almost-cycle condition \eqref{simplealmostcycles} in Theorem \ref{newsufficientconditionstorsiongrowth} replaces the small eigenvalue condition \eqref{fewsmalleigenvalues} from Theorem \ref{sufficientconditionstorsiongrowth}.  Additionally, the simple almost-cycle condition \eqref{simplealmostcycles} from Theorem \ref{newsufficientconditionstorsiongrowth} distinctly resembles the simple cycle condition from Theorem \eqref{sufficientconditionstorsiongrowth}. 

\begin{remark}
Let $M_0$ be a closed hyperbolic 3-manifold and $M \rightarrow M_0$ a finite cover.  Make the same assumptions as in the statement of Theorem \ref{newsufficientconditionstorsiongrowth}.  The simple almost-cycle condition \eqref{simplealmostcycles} implies that 
\begin{equation} \label{polynomialspectralgap}
\frac{1}{\lambda_1^1(M)} = O_{M_0} \left( \vol(M)^C \right).
\end{equation} 
\eqref{polynomialspectralgap} is not implied by the small eigenvalue condition \eqref{fewsmalleigenvalues}.  However, given the best progress to date on the small almost-betti number problem \cite{sarnaklettertorudnick}, it is difficult to imagine proving \eqref{fewsmalleigenvalues} without also proving a spectral gap of quality similar to \eqref{polynomialspectralgap}.  Unfortunately, 
\begin{equation} \label{reallybadspectralgap}
\text{multiplicity of the 1-form eigenvalue } \lambda_1^1(M) \cdot \log \left( \frac{1}{\lambda_1^1(M)} \right) = O_{M_0} \left( \vol(M) \right)
\end{equation}
gives the best currently known lower bound for $\lambda_1^1(M).$  The gulf between \eqref{polynomialspectralgap} and \eqref{reallybadspectralgap} 
is enormous.  We will revisit the issue of deriving improved lower bounds for $\lambda_1^1(M),$ or equivalently constructing simple almost cycles, in \S \ref{1formshighdimension} and in future work.
\end{remark}

\subsection{Outline}
Let $M_0$ be a closed hyperbolic manifold and $M \rightarrow M_0$ an arbitrary finite cover satisfying $b_1(M) = 0.$  Let $K_0$ be a triangulation of $M_0$ and $K$ the pullback triangulation of $M.$

\begin{itemize}
\item
In \S \ref{sobolevestimate}, we recall standard Sobolev estimates needed in \S \ref{newlowerboundlambda1}, \S \ref{almostprimitivesandregulators} and \S \ref{lambda1whitneylambda1derham}.

\item
\S \ref{newlowerboundlambda1} is the heart of this paper, building toward the key Corollary \ref{lowerboundeigenvalue}.  We construct almost-primitives for Laplacian eigen 1-forms on the image of $d^\ast$ on $M$ when $b_1(M) = 0.$  If the eigenvalue is extremely small, this almost succeeds.  On the other hand, the image of $d$ and the image of $d^\ast$ are orthogonal.  This tension results in lower bounds for $\lambda_1^1(M)_{d^\ast}.$

In \S \ref{geometrytwotypesfundamentaldomains}, we control the geometry of two types of fundamental domains for $M$ insofar as necessary to estimate terms arising in Corollary \ref{lowerboundeigenvalue}.

\item
\S \ref{almostprimitivesandregulators} describes how to extend the results of \S \ref{newlowerboundlambda1} when $b_1(M) > 0.$  In particular, our lower bounds for $\lambda_1^1(M)_{d^\ast}$ are as good when $b_1(M) = 1$ as they are when $b_1(M) = 0,$ \emph{independent of the 1-form regulator of $M$}.

\item
In \S \ref{whitneycomplex}, we compare combinatorial and Riemannian $L^p$-norms on the cochain complex $C^\bullet(M;K).$  

\item
In \S \ref{whitneylaplacianscl}, we prove that $\frac{1}{\lambda_{1,\mathrm{Whitney}}^1(M)_{d^\ast}}$ controls the (stable) area of surfaces bounding loops in $M.$

\item
In \S \ref{lambda1whitneylambda1derham}, we prove that either $\frac{1}{\lambda_1^1(M)_{d^\ast}} = O_{M_0} \left( \vol(M)^{-1} \right)$ or there is a comparison,
$$\lambda_1^1(M)_{d^\ast} = O_{M_0} \left( \vol(M) \cdot \lambda_{1, \mathrm{Whitney}}^1(M)_{d^\ast} \right).$$

\item
In \S \ref{applications}, we show how our main results imply an exponential upper bound
$$\frac{1}{\lambda_1^1(M)_{d^\ast}} \leq \exp( O_{M_0}( \vol(M))).$$
As we explain, it is often possible to prove a \emph{polynomial} upper bound for $\frac{1}{\lambda_1^1(N)_{d^\ast}}$ for hyperbolic $n$-manifolds $N$ when $n > 3.$  This gives a new prospect for proving useful upper bounds for $\frac{1}{\lambda_1^1(M)_{d^\ast}},$ for \emph{hyperbolic 3-manifolds} $M,$ by geodesically embedding $M$ in a higher dimensional hyperbolic manifold $N$ and applying retraction theorems such as those from \cite{BHW}.

\item 
In \S \ref{alternatelowerboundlambda10}, we show that if $M$ is a closed hyperbolic $n$-manifold, then
$$\frac{1}{\lambda_1^1(M)_d} \leq C \cdot \vol(M) \cdot \mathrm{diam}(M)^2$$
for some constant $C$ depending only on a lower bound for the injectivity radius of $M.$
\end{itemize}

\subsection{Acknowledgements}
We would like to thank Nicolas Bergeron, Danny Calegari, Nathan Dunfield, Aurel Page, Peter Sarnak, Akshay Venkatesh, Alden Walker, and Alex Wright for stimulating discussions related to the present work.      

\section{$L^\infty$ Estimates} \label{sobolevestimate}
The Sobolev inequality for $\mathbb{H}^n$ gives for all $\phi\in C_0^\infty(\mathbb{H}^n),$
\begin{equation}\label{sobolev}
\int_{\mathbb{H}^n}|d\phi|^2dv\geq \kappa_n \left(\int_{\mathbb{H}^n}|\phi|^{\frac{2n}{n-2}}dv \right)^{\frac{n-2}{n}},
\end{equation}
where $$\kappa_n:= \frac{ n(n-2)\vol(S^n)^{2/n}}{4}.$$   
See, for example, \cite[Section 8.2]{Hebey}.

 Let $M$ be a compact hyperbolic $n$-manifold with injectivity radius $D$. 
\begin{proposition} \label{supnormeigenfunction}
Let $f$ be a $q$-form on $M$ with 
 $\Delta f = \lambda f.$ For all $L<\frac{D}{2}$, there exists $C(n,q,L,\lambda)>0$ such that 
\begin{equation}\label{inftyest}\|f\|_{L^\infty(B(p,L))}\leq C(n,q,L,\lambda)\|f\|_{L^2(B(p,2L))}.\end{equation}

For every fixed $X,Y,$ the constant $C(n,q,L,\lambda)$ is uniformly bounded above for $\lambda \leq X$ and $L \geq Y.$     
\end{proposition}
 
\begin{proof} The existence of such an estimate is an immediate consequence of the Sobolev embedding theorem. 
For the convenience of the reader and to determine the dependence of $C(n,q,L,\lambda)$ on $L$ and $\lambda$, we recall a standard Moser iteration argument leading to (\ref{inftyest}). Let $\eta:\IR\to [0,1]$ be a smooth function identically $1$ on $(-\infty,L ]$ and supported on $(-\infty,2L]$, with $|d\eta|\leq \frac{2}{L}$. Set $\eta_k(t) = \eta(2^k(t-L))$. Observe $\eta_k(t)=1$ on  $(-\infty,L(1+2^{-k})  ]$ and supported on $(-\infty,L(1+2^{1-k})]$.   Let $\chi_k$ denote the characteristic function of $B(p,L(1+2^{-k}))$. Then 

\begin{equation*}
|d\eta_k|\leq \frac{2^{k+1}}{L}\eta_{k-1}.
\end{equation*}

Consider the Bochner formula for $f:$

\begin{equation}\label{boch1}
\Delta f = \nabla^*\nabla f - q(n-q)f = \lambda f.
\end{equation}

Taking the $L^2$-inner product of $\Delta f$ with $\psi^2f$ for a smooth function $\psi$ in \eqref{boch1} and integrating by parts gives 
\begin{equation}\label{agboch} 
\|\nabla (\psi f)\|^2 - \langle (q(n-q)+\lambda)\psi f,\psi f\rangle  = \||d\psi|f\|^2.
\end{equation}
Recall Kato's inequality: 
\begin{equation}\label{kato} 
\|\nabla F\|^2 \geq \||d |F|\|^2.
\end{equation}

Inserting this into (\ref{agboch}) yields 
\begin{equation}\label{agbochest} 
\|d |\psi f|\|^2 \leq (q(n-q)+\lambda)\| \psi f\|^2  + \||d\psi|f\|^2.
\end{equation}

Now we choose $$\psi = \psi_k = \eta_k(r)|f|^{\gamma_k-1},$$
$\gamma_k>1$ to be determined.  Substituting this choice of $\psi$ into \eqref{agbochest} gives  
\begin{align*}
\|d(\psi_k |f|)\|^2 &\leq  (q(n-q)+\lambda)\| \psi_k f\|^2  + \||d\psi_k|f\|^2 \\
&= (q(n-q)+\lambda)\| \psi_k f\|^2  + \|  d(\eta_k \frac{\gamma_k-1}{\gamma_k} |f|^{\gamma_k})  + \frac{1}{\gamma_k} |f|^{\gamma_k}d\eta_k  \|^2.
\end{align*}

Hence 
\begin{align*} 
\frac{2\gamma_k-1}{ \gamma_k^2}\|d(\psi_k |f|)\|^2 &\leq (q(n-q)+\lambda)\| \psi_k f\|^2  + \|  \frac{1}{\gamma_k} |f|^{\gamma_k}d\eta_k  \|^2 +2\langle  \frac{\gamma_k-1}{\gamma_k} d(\eta_k  |f|^{\gamma_k}) , \frac{1}{\gamma_k} |f|^{\gamma_k}d\eta_k  \rangle \\
&\leq  (q(n-q)+\lambda)\| \psi_k f\|^2  + \frac{1}{ \gamma_k  }\|   |f|^{\gamma_k}d\eta_k  \|^2
+ \frac{\gamma_k-1}{\gamma_k^2}  \|d(\eta_k  |f|^{\gamma_k})\|^2.
\end{align*}

Therefore
\begin{equation}  \label{anotherinequality}
\|d(\psi_k |f|)\|^2  \leq  (q(n-q)+\lambda )\gamma_k\|  \psi_k f\|^2   + \left|\left| \frac{2^{k+1}}{L}   |f|^{\gamma_k}\chi_{k-1}  \right|\right|^2.
\end{equation}
Applying \eqref{sobolev} to the left side of \eqref{anotherinequality} gives 
\begin{equation} \label{basicest} 
\kappa_n\|    |\eta_k|f|^{\gamma_k}|^{\frac{2n}{n-2}} \|^{\frac{n-2}{n}}   \leq  (q(n-q)+\lambda )\gamma_k\|  \psi_k f\|^2   +
\left|\left|\frac{2^{k+1}}{L}   |f|^{\gamma_k}\eta_{k-1}  \right|\right|^2.
\end{equation}

Set $\gamma : = \frac{n}{n-2}$ and $\gamma_k := \gamma_{k-1}\gamma = \gamma^k\gamma_0,$
then take $\gamma^k$ roots in \eqref{basicest} to get 

\begin{equation} \label{betterest} 
\|    \chi_{k} f \|_{L^{2\gamma^{k+1}}}^{2}   \leq 
\kappa_n^{-{\frac{1}{\gamma^k}}} \left[(q(n-q)+\lambda )\gamma^k+\frac{4^{k+1}}{L^2} \right]^{\frac{1}{\gamma^k}}
\| \chi_{k-1}f   \|_{L^{2\gamma^k}}^2.
\end{equation} 

Taking the product of \eqref{betterest} from $k = 0,\ldots.K$ and letting $K \to \infty$ gives  
\begin{equation} \label{inftyrest} 
\|      f \|_{L^{\infty}(B(p,L))}^{2}   \leq C(n,q,L,\lambda)\| f   \|_{L^2(B(p,2L))}^2,
\end{equation}
where 
\begin{equation*} \label{inftyconst} C(n,q,L,\lambda) =  \prod_{k=0}^\infty\kappa_n^{-{\frac{1}{\gamma^k}}} \left[(q(n-q)+\lambda )\gamma^k+\frac{4^{k+1}}{L^2} \right]^{\frac{1}{\gamma^k}}.
\end{equation*}

More generally, the same proof shows that if $\sigma$ is a section of a vector bundle $E$, $W$ a section of $\End(E)$, and $(\nabla^*\nabla + W)\sigma =0$, then 
\begin{equation*} \label{geninftyrest} 
\| \sigma \|_{L^{\infty}(B(p,L))}^{2}   \leq C(n, L,W)\| \sigma   \|_{L^2(B(p,2L))}^2,
\end{equation*}
where 
\begin{equation*} \label{geninftyconst} 
C(n, L,W) =  \prod_{k=0}^\infty\kappa_n^{-{\frac{1}{\gamma^k}}} \left[ ||W||_{L^{\infty}}\gamma^k+\frac{4^{k+1}}{L^2} \right]^{\frac{1}{\gamma^k}}.
\end{equation*}
\end{proof}
When $n$ is understood, set 
\begin{equation}\label{clambda}C(\lambda):= C(n,1,\text{Inj}(M),\lambda).
\end{equation}

\section{Constructing almost-primitives to bound $\lambda_1^1(M)$ below } \label{newlowerboundlambda1}
Let $M$ be a closed hyperbolic $n$-manifold.  Let $\Gamma$ denote $\pi_1(M)$ realized as a group of deck transformations of $\mathbb{H}^n$. 
Let $F$ denote a fundamental domain for $M$ in $\mathbb{H}^n$.  By this, we mean

\begin{itemize}
\item[(a)] 
$F$ is a closed domain in $\mathbb{H}^n$ with piecewise smooth boundary.  
\item[(b)] 
$\Gamma F = \mathbb{H}^n$ 

\item[(c)] 
The open sets $\gamma \cdot \mathrm{int}(F), \gamma \in \Gamma,$ are pairwise disjoint.  

\item[(d)]
We can partition $\partial F$ into oriented smooth submanifolds with corners $\{ \Sigma_1,\Sigma_1',\ldots,\Sigma_J, \Sigma_J' \}$  
\begin{itemize}
\item[(i)]
whose union is $\partial F$ and whose interiors are pairwise disjoint,  
\item[(ii)] the interior of each $\Sigma_j$ and $\Sigma_j'$ projects isometrically to $M$ under the quotient map $\mathbb{H}^n \rightarrow \Gamma \backslash \mathbb{H}^n = M,$ and 
\item[(iii)]
there exist $\{\gamma_1,\cdots\gamma_J\}\subset \Gamma$ with $\gamma_j \Sigma_j = - \Sigma_j'$ for $j = 1,\ldots, J.$ 
\end{itemize}
\end{itemize}

We will refer to the $\Sigma_i,$ and $\Sigma_j'$ as \emph{faces}. 
 We will refer to each $\gamma_j$ as a \emph{face-pairing element}. Set \begin{align}\label{pfdef}P_F:=\{\gamma_1,\cdots\gamma_J\}.\end{align} 

For $x,y \in \mathbb{H}^n,$ let $\alpha_{x,y}$ denote the oriented geodesic segment from $x$ to $y.$

\begin{lemma} \label{almostprimitive1formhyperbolicspace}
Let $b(q) = \int_{\alpha_{p,q}} f.$  For every vector $v \in T_q \mathbb{H}^n,$
$$|db_q(v) - f_q(v)| \leq \frac{1}{2} ||df||_{\infty} \cdot || v^\perp||,$$
where $v^\perp$ denotes the component of $v$ perpendicular to the tangent vector to $\alpha_{p,q}$ at $q.$  In particular,
$$||db - f||_{L^{\infty}(F)} \leq \frac{1}{2} ||df||_{L^{\infty}(M)}.$$
\end{lemma}

\begin{proof}
Let $\ell \subset T_q \mathbb{H}^n$ be the line tangent $\alpha_{p,q}$ at $q.$  

Suppose $v \in \ell.$  Then $db_q(v) - f_q(v) = 0.$

Suppose $v \in \ell^\perp.$  Let $\Delta_\epsilon$ be the geodesic triange with vertices $p,q,\exp_q(\epsilon v)$ and oriented boundary $\alpha_{p,q}, \alpha_{q,\exp_q(\epsilon v)},\alpha_{\exp_q(\epsilon v),p}.$  By Stokes,  
\begin{align} \label{directionalderivative}
\frac{1}{\epsilon}( b(\exp_q(\epsilon v)) - b(q)) + \frac{1}{\epsilon} \int_{\alpha_{q,\exp_q(\epsilon v)}} f &= \int_{\partial \Delta_\epsilon} f \nonumber \\
&= \frac{1}{\epsilon} \int_{\Delta_{\epsilon}} df. 
\end{align}

The area of a geodesic hyperbolic triangle with edges of lengths $a,b$ meeting at a right angle equals $\arctan \left(  \tanh \frac{a}{2} \cdot \tanh \frac{b}{2} \right).$  Letting $\epsilon \to 0$ in \eqref{directionalderivative} gives
$$| db_q(v) - f_q(v)| \leq ||df||_{\infty} \cdot \frac{||v||}{2} \cdot \tanh \left( \frac{d(p,q)}{2}  \right) \leq \frac{1}{2} ||df||_{\infty} \cdot ||v||.$$
The result follows.
\end{proof}

%
%
%

\begin{proposition}\label{newalmostprimitive}
Let $f$ be a 1-form on $M$ with $d^*f = 0.$  For every face $\Sigma_j$ in the aforementioned partition of $\partial F,$ fix some $q_j \in \Sigma_j.$  Then
\begin{equation*}
||f||_2^2 \leq \vol(\p F) \|f \|_{L^\infty} \left(  3\pi\|df\|_{L^\infty}
+\sup_j \left|\int_{\alpha_{q_j,\gamma_jq_j}}f \right| \right) + \frac{5}{2}\|df\|_{L^\infty}\|f\|_{L^2(M)}\sqrt{\vol(M)}.
\end{equation*}
\end{proposition}

\begin{proof} 
Lift $f$ to $\mathbb{H}^n$. For $q\in \bar F$, define 
\begin{equation}b(q) := \int_{\alpha_{p,q}}f,
\end{equation}
as in Lemma \ref{almostprimitive1formhyperbolicspace}

  Hence

\begin{align}\label{normfsq}
\|f\|^2 &=  \int_F f \wedge \ast f \nonumber \\
&= \int_F db \wedge \ast f + \int_F (f - db) \wedge \ast f \nonumber \\
&= \int_{\p F}b \ast f + \int_F (f - db) \wedge \ast f \nonumber \\
\end{align}
Write
\begin{equation}\begin{split}& \int_{\p F}b \ast f  =\sum_{j=1}^J\int_{\Sigma_j-\gamma_j\Sigma_j}b \ast f
=\sum_{j=1}^J\int_{\Sigma_j } \left(\int_{\alpha_{p,q}+ \alpha_{\gamma_jq,p}}f \right) \ast f.
\end{split}\end{equation}
Let $\Delta_{a,b,c}$ denote the oriented hyperbolic triangle with vertices $a,b,c$ and orientation such that $\p \Delta_{a,b,c}=\alpha_{ab}+\alpha_{bc}+\alpha_{ca}$.   For $q\in \Sigma_j$, let 
\begin{equation}R_{j,q}:= \Delta_{p,q,q_j}\cup \Delta_{p,\gamma q_j,\gamma q }\cup \Delta_{p, q_j,\gamma q_j }.\end{equation}  
Then 
\begin{equation}\int_{\alpha_{ \gamma_jq,p}+\alpha_{ p,q}}f  = - \int_{\alpha_{q_j,\gamma_jq_j}}f + \int_{R_{j,q}}df.
\end{equation}
 Here we have used $\int_{\alpha_{  q,q_j}+\alpha_{ \gamma q_j,\gamma q}}f = 0.$ 
  This gives 
\begin{equation}\label{thisgives}\begin{split}  \int_{\p F}b \ast f  
&  =\sum_{j=1}^J \left(- \int_{\alpha_{q_j,\gamma_jq_j}}f \right) \int_{\Sigma_j } \ast f+\sum_{j=1}^J\int_{\Sigma_j } \left(\int_{  R_{j,q}}df \right) \ast f
.\\ 
\end{split}\end{equation}
Since hyperbolic triangles have area at most $\pi$, we have 
\begin{equation} \left| \int_{  R_{j,q}}df \right|\leq 3\pi\|df\|_{L^\infty}.
\end{equation}

Substituting \eqref{thisgives} into \eqref{normfsq} and estimating gives  
\begin{align}\label{normfsqb} 
\|f\|_{L^2}^2 &\leq 3\pi\|df\|_{L^\infty}\sum_{j=1}^J\int_{\Sigma_j } | f|dA
+\sum_{j=1}^J \left| \int_{\alpha_{q_j,\gamma_jq_j}}f \right| \cdot \left| \int_{\Sigma_j } \ast f \right| +  || (f - db) \wedge \ast f ||_{L^1(F)}  \nonumber \\
&\leq \vol(\p F) \|f \|_{L^\infty} \left(  3\pi\|df\|_{L^\infty}
+\sup_j \left|\int_{\alpha_{q_j,\gamma_jq_j}}f \right| \right) + || (f - db) \wedge \ast f ||_{L^2(F)} \cdot \sqrt{\vol(M)}  \nonumber \\
&\leq \vol(\p F) \|f \|_{L^\infty} \left(  3\pi\|df\|_{L^\infty}
+\sup_j \left|\int_{\alpha_{q_j,\gamma_jq_j}}f \right| \right) + || f - db ||_{L^{\infty}(F)} \cdot ||f ||_{L^2(M)} \cdot \sqrt{\vol(M)}  \nonumber \\
&\leq \vol(\p F) \|f \|_{L^\infty} \left(  3\pi\|df\|_{L^\infty}
+\sup_j \left|\int_{\alpha_{q_j,\gamma_jq_j}}f \right| \right) + \frac{1}{2}|| df ||_{L^{\infty}(M)}   \cdot ||f ||_{L^2(M)} \cdot \sqrt{\vol(M)},  \nonumber\\
\end{align}
where the last line follows from Lemma \ref{almostprimitive1formhyperbolicspace}.

\end{proof}

\begin{corollary} \label{lowerboundeigenvalue}
Let $f$ be a 1-form on $f$ satisfying $d^\ast f = 0.$  Suppose that $f$ is a linear combination of eigen 1-forms of eigenvalue at most $\lambda.$  Let $q_j \in \Sigma_j$ be the fixed reference points chosen in Proposition \ref{newalmostprimitive}. Then 

\begin{align} \label{lowerboundeigenvalueequation}
 ||f||_2 &\leq \sqrt{\lambda} \cdot \vol(\p F) \cdot C(\lambda)^2 \cdot 3\pi \cdot ||f||_2  
+ C( \lambda) \cdot \sup_j \left| \int_{\alpha_{q_j, \gamma_j q_j}} f \right|  \nonumber \\
&+ \sqrt{\lambda} \cdot   C(\lambda) \cdot \frac{1}{2} \cdot \sqrt{\vol(M)} \cdot ||f||_2 . \nonumber\\
\end{align}

In particular, if some multiple of $\gamma_j$ bounds for every $j,$ then

\begin{align*}
\frac{1}{\sqrt{\lambda}} &\leq C(\lambda)^2(3\pi\cdot\vol(\p F)   +   \sup_j  \mathrm{sArea}(\gamma_j))   +   C( \lambda) \cdot \frac{1}{2} \cdot \sqrt{\vol(M)}. \nonumber\\
\end{align*}

The quantity $\mathrm{sArea}(\gamma)$ is defined in \eqref{sarea}, and  $C(\lambda)$ is defined in \eqref{clambda}.  
\end{corollary}
\begin{proof}
The first part follows upon applying the Sobolev inequality from Proposition \ref{supnormeigenfunction} to Proposition \ref{newalmostprimitive}.

For the second part:  suppose $\gamma^m$ is bounded by a surface $S.$  By Stokes theorem,
\begin{align*}
\left| \int_{\alpha_{q,\gamma q}} f \right| &= \frac{1}{m} \left| \int_{\alpha_{q,\gamma^m q}} f \right| = \frac{1}{m} \left| \int_S df \right| \\
&\leq ||df ||_{\infty} \cdot \frac{ \mathrm{area}(S)}{m}.
\end{align*}
Applying this inequality to all period integrals appearing on the right side of the inequality from Proposition \ref{newalmostprimitive} together with the Sobolev inequality \ref{supnormeigenfunction} gives the second part of the Corollary.     
\end{proof}

\subsection{The geometry of two types of fundamental domains} \label{geometrytwotypesfundamentaldomains}
In this section, we analyze the geometry of two types of fundamental domains $F.$  Understanding this geometry is necessary to control the $\vol(\partial F)$ terms occurring in the estimates in Corollary \ref{lowerboundeigenvalue}.  We also need to control $d(q_j,\gamma_jq_j)$ in order to bound the periods $\int_{q_j, \gamma_j q_j} f$.  Upper bounds on $\mathrm{diam}(F)$ suffice.

\subsubsection{Type 1: tree-type fundamental domains induced from a covering map}
 Let $M_0$ be a closed hyperbolic $n$-manifold.  Let $M \rightarrow M_0$ be a covering.  Let $F_0$ be a (closed) Dirichlet fundamental domain for $M_0$ relative to a fixed center $p_0 \in \mathbb{H}^n.$  Let $\Gamma$ and $\Gamma_0$ respectively denote $\pi_1(M)$ and $\pi_1(M_0).$  Let 

$$S_0 = \{  \gamma \in \Gamma_0: \gamma F_0 \cap F_0 \neq \emptyset \}.$$

Note that $S_0$ is a symmetric generating set for $\pi_1(M_0).$  The fundamental domain $F_0$ induces a tiling of $\mathbb{H}^n.$  Let $\mathcal{G}(M_0)$ denote the dual graph of this tiling.  Quotienting this tiling by $\Gamma$ induces a tiling of $M.$  The dual graph of the induced tiling of $M,$ which equals $\Gamma \backslash \mathcal{G}(M_0),$ is isomorphic to the Schreier graph $\mathcal{G}(M,M_0)$ of $M$ relative to $M_0$: vertices are given by elements of $\Gamma \backslash \Gamma_0$ and two vertices are connected by an edge if they differ by $s \in S_0$ \cite[Corollary 0.9]{Breuillard}. 

Let $T$ be a spanning tree in $\mathcal{G}(M,M_0).$ Fix a vertex $v_0 \in T.$  Associated with the unique geodesic in $T$ from $v_0$ to $v$ is a corresponding ordered sequence of elements $s_1,\ldots,s_n \in S_0$; these are the Schreier graph edge labels in the ordered edge sequence determined by the geodesic from $v$ to $v_0.$  Let $\gamma_{v_0,v} = s_n s_{n-1} \cdots s_1.$

\begin{define}
The \emph{tree-type fundamental domain} $F_T$ associated with $F_0$ and $T$ is 
$$F_T := \bigcup_{\text{vertices } v \text{ of } T} F_v, \text{ where } F_v:=  \gamma_{v_0,v}F_0.$$
\end{define}

The boundary $\partial F_T$ is a union of $\Gamma$-translates of codimension-1 faces of $F_0$ which project isometrically to $M$ and which have disjoint interiors.  Because $M$ is closed, these boundary faces can be identified in pairs; i.e. there  exists a decomposition of the boundary $F_T$ into $F_0$-faces $\Sigma_1,\Sigma_1',\ldots,\Sigma_J,\Sigma_J'$ and corresponding $\gamma_1,\ldots,\gamma_J \in \Gamma$ for which $\gamma_j \Sigma_j = - \Sigma_j'.$  

Thus, $F_T$ is a fundamental domain as defined at the beginning of \S \ref{newlowerboundlambda1}.  It is not in general convex, but we do not require convexity for the arguments of this section.  

\begin{lemma}\label{treetypeareabound}
The boundary volume $\vol(\partial F_T)$ is bounded above by 
$$\vol(\partial F_T) < \vol(\partial F_0) \cdot \frac{\vol(M)}{\vol(M_0)}.$$
\end{lemma}

\begin{proof}
This follows because $F_T$ is a union of $\frac{\vol(M)}{\vol(M_0)}$ copies of $F.$  
\end{proof}

\begin{lemma} \label{treetypediameterbound}
The diameter of $F_T$ is bounded above by
$$\mathrm{diam}(F_T) \leq \mathrm{diam}(F_0) \cdot \left( \mathrm{diam}(T) + 1 \right),$$
where $\mathrm{diam}(T)$ denotes the combinatorial diameter of the tree $T.$
\end{lemma}

\begin{proof}
Let $p \in F_{a}$ and $q \in F_{b},$ where  $a, b$ are vertices of $T.$  
Let $a = v_0 \to v_1 \to \cdots \to v_n =b$ be the unique shortest path from $a$ to $b.$  Let $F_{v_0},\ldots,F_{v_n}$ be the corresponding chain of $F_0$-tiles in $F_T.$  Every point $p_i \in F_{v_i}$ lies within $\mathrm{diam}(F_{v_i}) = \mathrm{diam}(F_0)$ of some point $p_{i+1}$ of $F_{i+1}.$  The broken geodesic path $p = p_0 \to p_1 \to \cdots \to p_n = q$ therefore has total length at most $(n+1) \cdot \mathrm{diam}(F_0) \leq (\mathrm{diam}(T) + 1) \cdot \mathrm{diam}(F_0).$  Therefore, $d(p,q) \leq (\mathrm{diam}(T)+1) \cdot \mathrm{diam}(F_0).$ 
\end{proof}

For a special choice of tree $T,$ the diameter $\mathrm{diam}(T)$ can be bounded above in terms of $\mathrm{diam}(\mathcal{G}(M,M_0)).$

\begin{lemma} \label{fattree}
Let $G$ be an arbitrary finite, connected graph.  There exists a spanning subtree $T_{\mathrm{fat}} \subset G$ for which $\mathrm{diam}(T_{\mathrm{fat}}) \leq 2 \mathrm{diam} (G).$
\end{lemma}

\begin{proof}
Fix a vertex $v \in G.$ A {\em shortest path subtree relative to $v$} is a spanning subtree, $T_{\mathrm{fat}}$, such that for all $w \in T_{\mathrm{fat}}, d_{T_{\mathrm{fat}}}(w,v) = d_G(w,v).$  There exists at least one  shortest path  subtree relative to $v.$   For the reader's convenience, we recall one such construction \cite[\S 3]{shortestpathtree}:  to every vertex $w  \in G\setminus\{v\},$ assign a neighboring vertex $p_w \in G$ for which $d_G(p_w,v) = d_G(w,v) - 1.$  The subgraph with full vertex set and edge set the edges connecting $w$ and $p_w$ for every $w \neq v$ is a shortest path tree relative to $v.$ 

In particular, for every $a,b \in T_{\mathrm{fat}},$
\begin{align*}
d_{T_{\mathrm{fat}}}(a,b) &\leq d_{T_{\mathrm{fat}}}(a,v) + d_{T_{\mathrm{fat}}}(v,b) \\
&= d_G(a,v) + d_G(v,b) \\
&\leq 2 \mathrm{diam}(G).
\end{align*}     
\end{proof}

\begin{define}
  Let $v$ be a vertex of $\mathcal{G}(M_0).$  Define the \emph{combinatorial ball of radius $r$ centered at $F_v$}, denoted $B_{\mathrm{comb},r}(F_v),$ to be
$$B_{\mathrm{comb},r}(F_v) = \bigcup_{d_{\mathcal{G}(M_0)}(w,v) \leq r} F_w \subset \mathbb{H}^n.$$
 Define the \emph{combinatorial sphere of radius $r$ centered at $F_v$}, denoted $S_{\mathrm{comb},r}(F_v),$ to be 
$$S_{\mathrm{comb},r}(F_v) := \partial B_{\mathrm{comb},r}(F_v).$$
\end{define}

The next lemma bounds $\mathrm{diam}(\mathcal{G}(M,M_0))$ above in terms of $\mathrm{diam}(M).$  

\begin{lemma} \label{diametercomparison}
Suppose that $B_{\mathrm{comb},1}(F_v)$ projects isometrically to $M,$ for all $v$. Let $r_0 := d(F_0, S_{\mathrm{comb},1}(F_0)).$  Let $k_0$ denote the number of $F_0$-tiles in $B_{\mathrm{comb},1}(F_0).$  Suppose that $a \in \mathrm{int}(F_{v_1}), b \in \mathrm{int}(F_{v_2}).$   Suppose $d(a,b) < r_0.$  Then    
$$d_{\mathcal{G}(M,M_0)}(v_1,v_2) \leq 2 k_0.$$

In particular, for every $p,q \in M,$ if $p \in F_{w_1}, q \in F_{w_2}$, then
$$d_{\mathcal{G}(M,M_0)}(w_1,w_2) \leq \frac{d(p,q)}{r_0} \cdot 2k_0 + 2k_0$$
and hence
$$\mathrm{diam}(\mathcal{G}(M,M_0)) \leq \frac{\mathrm{diam}(M)}{r_0} \cdot 2k_0 + 2k_0.$$
\end{lemma}

\begin{proof}
For all $v$, $B_{\mathrm{comb},1}(F_v)$ is isometric to $B_{\mathrm{comb},1}(F_0)$. Hence $r_0$ and $k_0$ are the same for all balls. 
Because $d(a,b) < r_0,$ the geodesic segment $\alpha_{a,b}$ intersects only those $F_0$-tiles of $M$ contained in $B_{\mathrm{comb},1}(F_{v_1})$  Perturbing $a,b$ as necessary, we may assume that $\alpha_{a,b}$ intersects only codimension-1 faces of the $F_0$-tiles.    

The combinatorial distance $d_{\mathcal{G}(M,M_0)}(v_1,v_2)$ is at most the number of intersection points between codimension 1 faces of tiles of $B_{\mathrm{comb},1}(F_v)$ and the segment $\alpha_{a,b}.$  By convexity, $\alpha_{a,b}$ intersects the boundary of every tile at most twice.  Therefore, 
$$d_{\mathcal{G}(M,M_0)}(v_1,v_2) \leq 2 k_0$$
as claimed.

For the second claim, divide a length minimizing geodesic from $p$ to $q$ into $m = \lfloor \frac{d(p,q)}{r_0 - \epsilon} \rfloor$ equal segments of length $r_0 - \epsilon$ together with one terminal segment of length $< r_0 -\epsilon.$  Applying the above argument to all $m+1$ segments gives

\begin{align*}
d_{\mathcal{G}(M,M_0)}(w_1,w_2) &\leq \left\lfloor \frac{d(p,q)}{r_0 - \epsilon} \right\rfloor \cdot 2k_0 + 2k_0 \\
&\leq \frac{d(p,q)}{r_0 - \epsilon} \cdot 2k_0 + 2k_0, \forall \epsilon>0,
\end{align*} 
and the result follows.
\end{proof}

\begin{corollary} \label{diameterboundtreetypefundamentaldomain}
There exists a spanning subtree $T \subset \mathcal{G}(M,M_0)$ for which
$$\mathrm{diam}(F_T) \leq \mathrm{diam}(F_0) \cdot \left(  \frac{4k_0}{r_0} \cdot \mathrm{diam}(M) + 4k_0 + 1 \right),$$
where $r_0$ and $k_0$ are the constants from Proposition \ref{diametercomparison}. 
\end{corollary}

\begin{proof}
Let $T = T_{\mathrm{fat}}$ be the spanning subtree from Lemma \ref{fattree}.  The Corollary follows upon combining Lemma \ref{fattree} and Lemma \ref{diametercomparison}.
\end{proof}

\subsubsection{Type 2: Dirichlet fundamental domains}
Let $N$ be a closed hyperbolic $n$-manifold with fundamental group $G = \pi_1(N)$ acting by deck transformations in $\mathbb{H}^n.$
\begin{define}
The \emph{Dirichlet fundamental domain domain $F$ associated to $N$ and $p_0 \in \mathbb{H}^n$} is
$F := \{ x \in \mathbb{H}^n : d(x,p_0) \leq d(x, \gamma p_0) \text{ for all } 1 \neq \gamma \in G \}.$
\end{define}

If $F$ is a Dirichlet domain for $N,$ the group $G$ is generated by the finite symmetric set
$$S = S_G := \{ \gamma \in G: \gamma F \cap F \neq \emptyset \}.$$

\begin{lemma} \label{upperbounddiameterdirichletdomain}
 There is an upper bound
$$\mathrm{diam}(F) \leq 2 \cdot \mathrm{diam}(M).$$ 
\end{lemma}

\begin{proof}
\ Suppose $x \in \mathbb{H}^n$ satisfies $d(x,p_0) > \mathrm{diam}(M).$  The projection of $B_{\mathrm{diam}(M)}(x)$ to $M$ covers all of $M.$  Therefore, there is some $\gamma \in \Gamma$ for which $d(x,\gamma p_0) < \mathrm{diam}(M).$  But by assumption, $d(x,p_0) > \mathrm{diam}(M).$  Thus $x \notin F.$  
It follows that $F$ is contained in $B_{\mathrm{diam}(M)}(p_0),$ implying that $\mathrm{diam}(F) \leq 2 \mathrm{diam}(M).$  
\end{proof}

The following Proposition proves an upper bound on the number of codimension-1 faces of $F.$  Combined with the diameter upper bound from Lemma \ref{upperbounddiameterdirichletdomain}, this yields an upper bound for $\vol(\partial F).$

\begin{proposition} \label{volumeRneighborhood}
 $$\vol(\partial F) \leq E_n \cdot e^{(2n-3) \cdot 4 \mathrm{diam}(M)},$$
for some constant $E_n$ depending only on $n.$ 
\end{proposition}

\begin{proof}
The cells of $\partial F$ are formed by intersections of bisectors: $B_{\gamma} := \{x: d(x,p_0) = d(x, \gamma p_0) \}.$  However, $d(x,p_0) \leq \mathrm{diam} F$ for every $x \in B_{\gamma}.$  Therefore, if the bisector $B_{\gamma}$ intersects $F,$ we must have 
$$d(p_0, \gamma p_0) \leq d(p_0, x) + d(x, \gamma p_0) =2d(x,p_0) \leq 2 \mathrm{diam}(F).$$
Therefore, 
\begin{align*}
\# (n-1)\text{-dimensional faces of } \partial F &\leq \# \{ \gamma \in \Gamma: d(p_0, \gamma p_0) \leq 2 \mathrm{diam} F \} \\
&\leq \frac{\vol(B_{2 \mathrm{diam}(F) + \frac{1}{2} \mathrm{inj(M)}})}{\vol(B_{\frac{1}{2} \mathrm{inj}(M)}) } \\
&\leq D_n e^{(n-1)2 \mathrm{diam}(F)}. 
\end{align*}
%
%
Hence
\begin{align} \label{firstvolumeupperbound}
\vol(\partial F) &\leq \# (n-1)\text{-dimensional faces of } \partial F \cdot \max_{C = \text{ codimension-}1 \text{ face of } \partial F} \vol(C) \nonumber \\
&\leq D_n e^{(n-1)2 \mathrm{diam}(F)} \cdot \vol (B_{n-1}(2\mathrm{diam} F)) \nonumber \\
&\leq E_n \cdot e^{(n-1) 2 \mathrm{diam}(F)} \cdot e^{(n-2) 2\mathrm{diam} F} \nonumber \\
&\leq E_n \cdot e^{(2n-3) \cdot 4 \mathrm{diam}(M)},
\end{align}
where the last line follows from Proposition \ref{upperbounddiameterdirichletdomain}.
\end{proof}

\begin{remark}
There is an upper bound for the diameter of a closed hyperbolic $n$-manifold $M$ of $\frac{1}{\lambda_1^0(M)} \log \vol(M).$  So in everything that preceded, upper bounds of $\exp( \mathrm{diam} (M))$ may all be replaced by $\vol(M)^{1 / \lambda_1^0(M)}.$  
\end{remark}

\section{Almost-primitives and regulators when $b_1(M) > 0$} \label{almostprimitivesandregulators}
Theorem \ref{bodyupperboundpositivebettinumber} proves a general upper bound for $\frac{1}{\lambda_1^1(M)_{d^\ast}},$ where $M$ is a closed hyperbolic $n$-manifold, in terms of stable area of an explicit subset of $\Gamma = \pi_1(M)$ whose projection to $H_1(M,\mathbb{Q})$ is trivial.  The key is to bound the ``period integrals" of a 1-form $f$ in the image of $d^\ast$ with smallest positive eigenvalue.  In the notation of Theorem \ref{bodyupperboundpositivebettinumber}, we bound the period integral over $\gamma,$ where $\gamma$-pairs two faces of a fundamental domain for $M,$ in terms of the stable area of $\gamma' = \gamma \cdot \gamma_1^{-n_1(\gamma)} \cdots \gamma_k^{-n_k(\gamma)}.$  

However, if $\gamma'$ is long, then its stable area should be bounded below below by $\text{constant} \cdot \ell / \log \ell,$ where $\ell = |n_1(\gamma)| + \cdots + |n_k(\gamma)|,$ with very high probability (cf. \cite[Conjecture A.10]{CW}).  It is thus imperative that we prove good upper bounds on the integers $n_i(\gamma),$ which are intersection numbers between $\gamma$ and surfaces $S_i$ generating $H_2(M,\mathbb{Z}).$  When $n=3,$ Proposition \ref{boundedgeometrysurface} uses known facts about minimal surfaces in hyperbolic 3-manifolds to represent every $S_i$ by homologous surfaces with ``bounded geometry."  Proposition \ref{boundingintersectionnumbers} bounds the intersection numbers $S_i \cap \gamma$ above in terms of $A',$ an upper bound for the minimal area representatives of every class $S_i.$  

When $n = 3$ and $b_1(M) = 1,$  Proposition \ref{regulatorindependentbounds} esimates the damping factor $\frac{||f||_2}{(||f||_2^2 + ||h||_2^2)^{1/2}}$ occuring in Theorem \ref{bodyupperboundpositivebettinumber}.  The upshot: this damping factor is smaller than the inverse of the translation length of $\gamma'.$  As a result, Proposition \ref{regulatorindependentbounds} proves ``regulator-independent" upper bounds on $\frac{1}{\lambda_1^1(M)_{d^\ast}},$ i.e. upper bounds in terms of $\mathrm{sArea}(\gamma') / \ell(\gamma'),$ for an explicit finite collection of $\gamma' \in \Gamma,$ which is independent of $A'.$   

\begin{remark}
Though Proposition \ref{regulatorindependentbounds} is proven only when $n = 3$ and $b_1(M) = 1,$ ``regulator-independent" upper bounds on $\frac{1}{\lambda_1^1(M)_{d^\ast}}$ are to be expected in general.  Indeed, the Cheeger-M\"{u}ller Theorem suggests that small eigenvalues and regulators suppress each other.  See \S \ref{motivation} for further discussion.
\end{remark}

\subsection{General upper bounds on $\frac{1}{\lambda_1^1(M)_{d^\ast}}$ when $b_1(M) > 0$}

Let $M \rightarrow M_0$ be a finite cover of the closed hyperbolic $n$-manifold $M_0.$   Suppose that $H_1(M,\mathbb{Z}) = \mathbb{Z} \langle \gamma_1 \rangle \oplus \cdots \oplus \mathbb{Z} \langle \gamma_k \rangle \oplus \mathrm{finite}.$  We may take $\gamma_1, \ldots, \gamma_k\subset P_F$ for some fundamental domain $F$ of $M$ in $\mathbb{H}^n$; we take this fundamental domain to be either a tree-type domain or a Dirichlet domain in the sense of \S \ref{geometrytwotypesfundamentaldomains}.  By Corollary \ref{diameterboundtreetypefundamentaldomain} in the case of tree-type fundamental domains and Lemma \ref{upperbounddiameterdirichletdomain} in the case of Dirichlet fundamental domains, 
$$\ell(\gamma_i) \leq \begin{cases} \mathrm{diam}(F_0) \cdot \left(  \frac{4k_0}{r_0} \cdot \mathrm{diam}(M) + 4k_0 + 1 \right) & \text{ if the fundamental domain is tree-type } \\ 2 \cdot \mathrm{diam}(M) &\text{ if  the fundamental domain is Dirichlet}  \end{cases}$$
for every $i.$

\begin{proposition} \label{subtractingharmonicpart}
Let $n_i = n_i(\gamma) \in \mathbb{Z}$ be the unique integers for which $\gamma - n_1 \gamma_1 - \cdots - n_k \gamma_k$ is torsion in $H_1(M,\mathbb{Z}).$

Suppose $f$ is a 1-form on $M$ satisfying $d^\ast f = 0.$  Fix a base point  $q_0 \in M.$  Let $h$ be a harmonic 1-form satisfying 
$$\int_{\alpha_{q_0,\gamma_i q_0}} f = \int_{\alpha_{q_0, \gamma_i q_0}} h \text{ for } i = 1,\ldots,k.$$

   Then
\begin{equation*}
\left| \int_{\alpha_{q_0, \gamma q_0}} (f-h) \right| \leq||df||_{\infty} \cdot \left(  \mathrm{sArea}(\alpha_{q_0, \gamma \cdot \gamma_1^{-n_1} \cdots \gamma_k^{-n_k} q_0}) + 5 b_1(M) \pi \right)
\end{equation*}
\end{proposition}

\begin{proof}
Suppose that $\alpha_{q_0, \left(\gamma  \cdot \prod_{i = 1}^k \gamma_i^{- n_i} \right)^m q_0}$ bounds $S.$  The geodesic triangle $\Delta$ with positively oriented vertex set $q_0, \gamma_1^{n_1} \cdots \gamma_k^{n_k} q_0, \gamma q_0$ has oriented boundary $\alpha_{q, \gamma_1^{n_1} \cdots \gamma_k^{n_k} q_0} + \alpha_{\gamma_1^{n_1} \cdots \gamma_k^{n_k} q_0, \gamma q_0 } + \alpha_{\gamma q_0, q_0}.$  By Stokes,

\begin{align} \label{subtractingtorsionfreepart}
\left| \int_{\alpha_{q_0, \gamma q_0}} (f-h) \right| &\leq \left| \int_{\alpha_{q_0, \gamma_1^{n_1} \cdots \gamma_k^{n_k} q_0}} (f-h) \right| + \left| \int_{\alpha_{\gamma_1^{n_1} \cdots \gamma_k^{n_k} q_0, \gamma q_0 }} (f-h) \right| + \left| \int_\Delta d(f-h) \right| \nonumber \\
&=  \left| \int_{\alpha_{q_0, \gamma_1^{n_1} \cdots \gamma_k^{n_k} q_0}} (f-h) \right| + \left| \int_{\alpha_{q_0, \gamma_k^{-n_k} \cdots \gamma_1^{-n_1}\gamma q_0 }} (f-h) \right| + \left| \int_\Delta df \right| \nonumber \\
&=  \left| \int_{\alpha_{q_0, \gamma_1^{n_1} \cdots \gamma_k^{n_k} q_0}} (f-h) \right| + \frac{1}{m} \left| \int_{\alpha_{q_0, \left(\gamma_k^{-n_k} \cdots \gamma_1^{-n_1}\gamma \right)^m q_0 }} (f-h) \right| + \left| \int_\Delta df \right| \nonumber \\
&=  \left| \int_{\alpha_{q_0, \gamma_1^{n_1} \cdots \gamma_k^{n_k} q_0}} (f-h) \right| + \frac{1}{m} \left| \int_S d(f-h) \right| + \left| \int_\Delta df \right| \nonumber \\
&\leq  \left| \int_{\alpha_{q_0, \gamma_1^{n_1} \cdots \gamma_k^{n_k} q_0}} (f-h) \right| + || df ||_{\infty} \cdot \left( \frac{\mathrm{area}(S)}{m} + \mathrm{area}(\Delta) \right). 
\end{align}

As \eqref{subtractingtorsionfreepart} holds for arbitrary $m,S$ for which $\alpha_{q_0, \left( \gamma \gamma_1^{-n_1} \cdot \gamma_k^{-n_k} \right)^m}$ bounds $S,$ it follows that

\begin{equation} \label{firstupperboundintermsofsarea}
\left| \int_{\alpha_{q_0, \gamma q_0}} (f-h) \right| \leq \left| \int_{\alpha_{q_0, \gamma_1^{n_1} \cdots \gamma_k^{n_k} q_0}} (f-h) \right| + ||df||_{\infty} \cdot \left( \mathrm{sArea}(\alpha_{q_0,\gamma \cdot \gamma_1^{-n_1} \cdots \gamma_k^{-n_k} q_0}) + \pi \right).
\end{equation}

Let $$\eta_{\ell} = \gamma_{k-l+1}^{n_{k-l+1}} \cdots \gamma_{k-1}^{n_{k-1}}\gamma_k^{n_k}.$$
    The geodesic segments $\alpha_{q_0, \gamma_k^{n_k} q_0}, \alpha_{\eta_1 q_0, \gamma_{k-1}^{n_{k-1}} \eta_1 q_0}, \ldots, \alpha_{\eta_{k-1} q_0, \gamma_1^{n_1} \eta_{k-1} q_0}, \alpha_{\gamma_1^{n_1} \cdots \gamma_k^{n_k} q_0, q_0}$ form the oriented boundary of a ``broken  geodesic $k+1$-gon" $P$, the union of $k-1$ geodesic triangles meeting at kinks.  This broken polygon has area at most $(k-1) \pi.$  By Stokes,

\begin{align} \label{brokenpolygon}
\left| \int_{\alpha_{q_0, \gamma_1^{n_1} \cdots \gamma_k^{n_k} q_0}} (f-h) \right| &\leq  \sum_i \left| \int_{\alpha_{\eta_i q_0, \gamma_{k-i}^{n_{k-i}} \eta_i q_0}} (f-h)  \right| + \left|\int_P d(f-h) \right| \nonumber \\ 
&=   \sum_i \left| \int_{\alpha_{q_0, \eta_i^{-1} \gamma_{k-i}^{n_{k-i}} \eta_i q_0}} (f-h)  \right| + \left| \int_P df \right| \nonumber \\
&\leq \sum_i \left| \int_{\alpha_{q_0, \eta_i^{-1} \gamma_{k-i}^{n_{k-i}} \eta_i q_0}} (f-h)  \right| + ||df||_{\infty} \cdot (k-1) \pi.
\end{align}

For arbitrary $a,b \in \Gamma,$ the geodesic segments $\alpha_{q_0, a^{-1} ba q_0} + \alpha_{a^{-1} b a q_0, b q_0} + \alpha_{b q_0, q_0}$ form the oriented boundary of a geodesic triangle $\Delta'.$  By Stokes,

\begin{align} \label{relatingconjugates}
\left| \int_{\alpha_{q_0, a^{-1} ba q_0}} (f-h) \right| &\leq \left|\int_{\alpha_{q_0,b q_0}} (f-h) \right| + \left| \int_{\alpha_{a^{-1} b a q_0, b q_0} } (f-h) \right| + \left|\int_{\Delta'} d(f-h) \right| \nonumber \\ 
&=  \left|\int_{\alpha_{q_0,b q_0}} (f-h) \right| + \left| \int_{\alpha_{p_0, [b,a^{-1}] p_0}} (f-h) \right| + \left| \int_{\Delta'} df \right| \text{ where } p_0 = a^{-1} ba q_0 \nonumber \\
\end{align}

Next, consider the oriented broken geodesic  $\alpha_{p_0, a^{-1}b^{-1}ap_0} + \alpha_{a^{-1}b^{-1}ap_0, b^{-1} ap_0} + \alpha_{b^{-1} a p_0,  a p_0} + \alpha_{ a p_0, [b,a^{-1}] p_0}+ \alpha_{  [b,a^{-1}] p_0,p_0}.$   It is the boundary of a broken geodesic pentagon $ P' $ which projects to a surface in $M$ with boundary $ - \alpha_{p_0, [b,a^{-1}] q_0}.$  Therefore,

\begin{align*}
\left| \int_{\alpha_{p_0, [b,a^{-1}] p_0}} (f-h) \right| &\leq \left| \int_{P'} df \right|  .\\
\end{align*}

Substituting back into \eqref{relatingconjugates} gives

\begin{align} \label{finalrelatingconjugates}
\left| \int_{\alpha_{q_0, a^{-1} ba q_0}} (f-h) \right| &\leq \left|\int_{\alpha_{q_0,b q_0}} (f-h) \right| + \left| \int_{P'} df \right| + \left| \int_{\Delta'} df \right| \nonumber \\ 
&\leq  \left|\int_{\alpha_{q_0,b q_0}} (f-h) \right| + ||df||_{\infty} \cdot \left( \mathrm{area}(P') + \mathrm{area}(\Delta') \right) \nonumber \\
&\leq   \left|\int_{\alpha_{q_0,b q_0}} (f-h) \right| + ||df||_{\infty} \cdot 4\pi. 
\end{align}

Substituting \eqref{finalrelatingconjugates} back into \eqref{brokenpolygon} gives

\begin{align} \label{firstsummand}
\left| \int_{\alpha_{q_0, \gamma_1^{n_1} \cdots \gamma_k^{n_k} q_0}} (f-h) \right| &\leq \sum_i  \left( \left| \int_{\alpha_{q_0, \gamma_{k-i}^{n_{k-i}} q_0}} (f-h) \right| + ||df||_{\infty} \cdot 4\pi \right) + ||df||_{\infty} \cdot (k-1)\pi \nonumber \\
&= \sum_i n_{k-i} \left|  \int_{\alpha_{q_0, \gamma_{k-i} q_0}} (f-h) \right| + ||df||_{\infty} \cdot (5k - 1) \pi \nonumber \\
&= \sum_i n_{k-i} \cdot 0 + ||df||_{\infty} \cdot (5k-1) \pi \nonumber \\
&= ||df||_{\infty} \cdot (5k-1) \pi.
\end{align}

Finally, substituting \eqref{firstsummand} back into \eqref{firstupperboundintermsofsarea} gives

\begin{align*}
\left| \int_{\alpha_{q_0, \gamma q_0}} (f-h) \right| &\leq ||df||_{\infty} \cdot (5k-1)\pi + ||df||_{\infty} \cdot \left( \mathrm{sArea}(\gamma \cdot \gamma_1^{-n_1} \cdots \gamma_k^{-n_k}) + \pi \right) \\
&= ||df||_{\infty} \cdot \left(  \mathrm{sArea}(\alpha_{q_0, \gamma \cdot \gamma_1^{-n_1} \cdots \gamma_k^{-n_k} q_0}) + 5 b_1(M) \pi \right)
\end{align*}
\end{proof}

\begin{proposition} \label{lowerboundeigenvaluepositivebettinumber}
Notation as in Corollary \ref{lowerboundeigenvalue}.  Let $M_0$ be a closed hyperbolic $n$-manifold.  Let $M \rightarrow M_0$ be a finite cover.  Suppose $H_1(M,\mathbb{Z}) = \mathbb{Z} \langle \gamma_1 \rangle \oplus \cdots \oplus \mathbb{Z} \langle \gamma_k \rangle \oplus \mathrm{finite}.$  For every $\gamma \in \Gamma,$ let $n_i(\gamma) \in \mathbb{Z}$ be the unique integers for which $\gamma - \sum_{i = 1}^k n_i(\gamma) \gamma_i \in H_1(M,\mathbb{Z})$ has finite order.  Fix a basepoint $q_0 \in M.$

Suppose $f$ is a 1-form on $M$ contained in the image of $d^\ast.$  Suppose that $h$ is a harmonic 1-form satisfying
$$\int_{\alpha_{q_0, \gamma_i q_0}} h = \int_{\alpha_{q_0, \gamma_i q_0}} f \text{ for } i = 1, \ldots, k.$$
Then

\begin{align*}
 \frac{1}{\sqrt{\lambda}} &\leq 3\pi\cdot \vol(\p F) \cdot C( \lambda)^2  +   C( \lambda) \cdot \frac{1}{2} \cdot \sqrt{\vol(M)}\nonumber \\
&+  C(  \lambda)^2 \cdot \frac{||f||_2}{ \left( ||f||_2^2 + ||h||_2^2 \right)^{1/2} } \cdot \left(  \sup_j  \mathrm{sArea}\left( \alpha_{q_0, \gamma_j \cdot \gamma_1^{-n_1(\gamma_j)} \cdots \gamma_k^{-n_k(\gamma_j)} q_0}  \right) + (5b_1(M) + 2) \pi  \right)  . \nonumber\\
\end{align*}



\end{proposition}

\begin{proof}
Let $q_j$ be the reference points used in Corollary \ref{lowerboundeigenvalue}.
The broken geodesic $\alpha_{q_j, \gamma_j q_j} + \alpha_{\gamma_j q_j, \gamma_j q_0} + \alpha_{\gamma_j q_0, q_0} + \alpha_{q_0,q_j}$ is the oriented boundary of a broken geodesic quadrilateral $Q$ of area at most $2\pi.$  The projection of this quadrilateral to $M$ has boundary $\alpha_{q_j, \gamma_j q_j} + \alpha_{\gamma_j q_0, q_0}.$  By Stokes and Proposition \ref{subtractingharmonicpart}:

\begin{align} \label{periodchangereferencepoint}
\sup_j \left| \int_{\alpha_{q_j, \gamma_j q_j}} (f-h) \right| &\leq \sup_j \left| \int_{\alpha_{q_0, \gamma_j q_0}} (f-h) \right| + \left| \int_Q d(f-h) \right| \nonumber \\
&\leq  \sup_j \left| \int_{\alpha_{q_0, \gamma_j q_0}} (f-h) \right| + ||df||_{\infty} \cdot 2\pi \nonumber \\
&\leq ||df||_{\infty} \cdot \left(  \sup_j  \mathrm{sArea}\left( \alpha_{q_0, \gamma_j \cdot \gamma_1^{-n_1(\gamma_j)} \cdots \gamma_k^{-n_k(\gamma_j)} q_0}  \right) + (5b_1(M) + 2) \pi  \right).
\end{align}

Substituting \eqref{periodchangereferencepoint} into the first inequality stated in Corollary \ref{lowerboundeigenvalue} (and remembering that $f-h$ here plays the role of $f$ in Corollary \ref{lowerboundeigenvalue}) followed by the Sobolev inequality from Proposition \ref{supnormeigenfunction} gives  

\begin{align*}
 \left( ||f||_2^2 + ||h||_2^2 \right)^{1/2} &\leq \sqrt{\lambda} \cdot \vol(\p F) \cdot C( \lambda)^2 \cdot 3\pi \cdot \left(||f||_2^2 + ||h||_2^2 \right)^{1/2} \\
&+ \sqrt{\lambda} \cdot C(  \lambda)^2 \cdot ||f||_2 \cdot \left(  \sup_j  \mathrm{sArea}\left( \alpha_{q_0, \gamma_j \cdot \gamma_1^{-n_1(\gamma_j)} \cdots \gamma_k^{-n_k(\gamma_j)} q_0}  \right) + (5b_1(M) + 2) \pi  \right)  \\
&+ \sqrt{\lambda} \cdot C( \lambda) \cdot \frac{1}{2} \cdot \sqrt{\vol(M)} \cdot \left( ||f||_2^2 + ||h||_2^2 \right)^{1/2} \nonumber \\
\end{align*} 

Upon rearranging:

\begin{align*}
 \frac{1}{\sqrt{\lambda}} &\leq 3\pi\cdot \vol(\p F) \cdot C( \lambda)^2  +   C( \lambda) \cdot \frac{1}{2} \cdot \sqrt{\vol(M)}\nonumber \\
&+  C(  \lambda)^2 \cdot \frac{||f||_2}{ \left( ||f||_2^2 + ||h||_2^2 \right)^{1/2} } \cdot \left(  \sup_j  \mathrm{sArea}\left( \alpha_{q_0, \gamma_j \cdot \gamma_1^{-n_1(\gamma_j)} \cdots \gamma_k^{-n_k(\gamma_j)} q_0}  \right) + (5b_1(M) + 2) \pi  \right)  . \nonumber\\
\end{align*} 
\end{proof}

\begin{theorem}[Geometric Upper Bound for $\frac{1}{\lambda_1^1(M)_{d^\ast}}$ when $b_1(M) > 0$] \label{bodyupperboundpositivebettinumber}
Let $M_0$ be a closed hyperbolic $n$-manifold.  Let $M \rightarrow M_0$ be a finite cover.  Suppose $H_1(M,\mathbb{Z}) = \mathbb{Z} \langle \gamma_1 \rangle \oplus \cdots \oplus\mathbb{Z} \langle \gamma_k \rangle \oplus \mathrm{finite}.$  Let $n_i(\gamma)$ be the unique integers for which $\gamma - \sum_{i = 1}^k n_i(\gamma) \gamma_i \in H_1(M,\mathbb{Z})$ has finite order.  Fix a basepoint $q_0 \in M.$

Let $f$ be a 1-form in the image of $d^\ast$ satisfying $\Delta f = \lambda f,$ where $\lambda = \lambda_1^1(M)_{d^\ast}.$  Let $h$ be the unique harmonic 1-form satisfying
$$\int_{\alpha_{q_0,\gamma_i q_0}} h = \int_{\alpha_{q_0,\gamma_i q_0}} f \text{ for } i = 1,\ldots,k.$$
Then
\begin{align*}
 \frac{1}{\sqrt{\lambda}} &\leq 3\pi\cdot C( \lambda)^2\cdot V   +   C( \lambda) \cdot \frac{1}{2} \cdot \sqrt{\vol(M)}  \nonumber \\
&+  C(  \lambda)^2 \cdot \frac{||f||_2}{ \left( ||f||_2^2 + ||h||_2^2 \right)^{1/2} } \cdot \left(  \sup_{\ell(\gamma) \leq D}  \mathrm{sArea}\left( \alpha_{q_0, \gamma \cdot \gamma_1^{-n_1(\gamma)} \cdots \gamma_k^{-n_k(\gamma)} q_0}  \right) + (5b_1(M) + 2) \pi  \right)  \\
&. \nonumber\\
\end{align*}    

where
$$V := \vol(\partial F) \leq \begin{cases} \vol( \partial F_0) \cdot \frac{\vol(M)}{\vol(M_0)} &\text{ if } $F$ \text{ is tree-type} \\ E_n \cdot e^{(2n-3) \cdot 4 \mathrm{diam}(M)} & \text{ if } F \text{ is Dirichlet} \end{cases}$$

and 
$$D := \max_j \ell(\gamma_j) \leq \begin{cases}  \mathrm{diam}(F_0) \cdot \left(  \frac{4k_0}{r_0} \cdot \mathrm{diam}(M) + 4k_0 + 1 \right)
&\text{ if } F \text{ is tree-type} \\ 2 \cdot \mathrm{diam}(M) &\text{ if } F \text{ is Dirichlet}.\end{cases}$$

Here, $F_0$ is a Dirichlet fundamental domain for $M_0, r_0$ and $k_0$ are the constants from Proposition \ref{diametercomparison}, $E_n$ is the constant from Proposition \ref{volumeRneighborhood}, and $C(\lambda)$ is the Sobolev constant from  \eqref{clambda}.  
\end{theorem} 

\begin{proof}
This follows immediately from Proposition \ref{lowerboundeigenvaluepositivebettinumber}.  The upper bounds on $V$ follow from Propositions \ref{treetypeareabound} and \ref{volumeRneighborhood}.  The upper bounds on $D$ follow from Propositions \ref{diameterboundtreetypefundamentaldomain} and \ref{upperbounddiameterdirichletdomain}.  
\end{proof}

The upper bound for $\frac{1}{\lambda_1^1(M)_{d^\ast}}$ from Theorem \ref{bodyupperboundpositivebettinumber} is only useful if $\gamma - \sum_{i = 1}^k n_i(\gamma) \gamma_i$ is a short loop for $\gamma$ appearing in the sum on the right side of the inequality featured therein.  We turn next to controlling the size of the $n_i(\gamma).$

\subsubsection{Controlling the free part of $\gamma \in \pi_1(M)$ via regulators} \label{freepartregulators}
Let $M_0$ be a closed hyperbolic $3$-manifold and let $M \rightarrow M_0$ be a finite cover.  Let $H_1(M,\mathbb{Z}) = \mathbb{Z} \langle \gamma_1 \rangle \oplus \cdot \oplus \mathbb{Z} \langle \gamma_k \rangle \oplus \mathrm{finite}$ and $H_2(M, \mathbb{Z}) = \mathbb{Z} \langle S_1 \rangle \oplus \cdots \oplus \mathbb{Z} \langle S_k \rangle \oplus \mathrm{finite}.$ 


Every $S_i$ is represented by a stable, properly embedded minimal surface of least area in its homology class \cite[5.1.6 and 5.4.6]{Federer} \cite[10.2]{Morgan} \cite[p. 28]{Simon}. Let $\Sigma = \Sigma_i$ denote such a surface. Schoen \cite[Theorem 3]{Schoen} proved that the second fundamental form of $\Sigma$ is bounded by some constant independent of $M$ and $\Sigma$. Using the Gauss equations for the curvature of submanifolds and the vanishing of the mean curvature of $\Sigma$, Schoen's bound implies that the curvature $K_g$ of $\Sigma$ with respect to the induced metric $g$ is bounded between $[-C,-1]$ for some constant $C \geq 1$ independent of $M, \Sigma.$ 
 
With the help of these curvature bounds, we construct a triangulation of $(\Sigma,g)$  of bounded geometry.  


%
 \begin{lemma}\label{muM}Let $M$ be a compact hyperbolic 3-manifold. Then there exists $\mu_M>0$ depending only on the injectivity radius of $M$ so that for every compact embedded stable minimal surface $S$ in $M$ the injectivity radius of $S$ is greater than or equal to $\mu_M$. 
\end{lemma}
\begin{proof} By \cite[Theorem 3]{Schoen} (See also \cite[Corollary 11]{Ros}), there exists $\mu_1>1$ depending only on the injectivity radius of $M$ so that the second fundamental form $\sigma$ of $S$ satisfies $|\sigma|<\mu_1$. Let $\gamma$ be a closed geodesic in $S$, parameterized by arclength. Then 
$$\nabla^N_{\dot\gamma}\dot \gamma = a(t)\nu(t),$$
for some scalar function $a$ and unit normal vector field $\nu$. Because $\gamma$ is parametrized by arclength, $$0=\frac{d^2}{dt^2}|\dot\gamma|^2,$$ and 
$$0 =  |\nabla^M_{\dot\gamma}\dot \gamma |^2 + \langle (\nabla^M_{\dot\gamma})^2\dot \gamma,\dot\gamma\rangle  = a^2-a\sigma(\dot\gamma,\dot\gamma).$$
Hence
\begin{equation}\label{torq}|\nabla^M_{\dot\gamma}\dot \gamma|\leq \mu_1.\end{equation}
In a geodesic ball $B_{\mathrm{inj}(M)} \subset M$ centered at $\gamma(0),$  $\lim_{t\to0}\langle\frac{\p}{\p r} , \dot \gamma(t)\rangle = 1$. 
Thus, 
\begin{align}r(\gamma(t)) &=\int_0^t \left\langle\dot\gamma(s),\frac{\p}{\p r} \right\rangle ds
=\int_0^t \left(1+\int_0^s\frac{d}{du} \left\langle\dot\gamma(u),\frac{\p}{\p r} \right\rangle du \right) ds \nonumber\\ 
& =t+\int_0^t\int_0^s \left\langle \nabla^M_{\dot\gamma}\dot\gamma(u),\frac{\p}{\p r} \right\rangle du ds +\int_0^t\int_0^s \left\langle\dot\gamma(u),\nabla^N_{\dot\gamma}\frac{\p}{\p r} \right\rangle du ds\nonumber \\
&\geq t-\mu_1\frac{t^2}{2}.
\end{align}
 
Hence, $r$ achieves a max greater than or equal to $ \min\{\frac{1}{2\mu_1},\mathrm{inj}(M)\}.$ The lower bound on $r$ gives a lower bound on the length of $\gamma$ and $\mathrm{inj}(S)$.  
The injectivity radius of $S$ is therefore greater than or equal to $\min\{\frac{1}{2\mu_1},\mathrm{inj}(M)\}.$ 
\end{proof} 
Let $V_r(-K)$ denote the volume of a geodesic ball of radius $r$ in the hyperbolic space of constant curvature $-K.$
\begin{proposition}[triangulations of $\Sigma$ with bounded geometry] \label{boundedgeometrysurface}
Suppose $\Sigma$ is a stable, properly embedded, minimal
 surface in hyperbolic 3-manifold $M.$  Suppose $\Sigma$ with its induced metric $g$ has curvature $K_g \in [-C, -1].$  Let $\mu_M$ denote the injectivity bound of Lemma \ref{muM}. There is a triangulation of $(\Sigma,g)$ by at most 
$$\frac{\vol_g(\Sigma)}{V_{\frac{\mu_M}{5}}(-1)} \cdot  \frac{V_{\mu_M}(-C)}{V_{\frac{\mu_M}{5}}(-1)}$$
triangles such that every triangle $T$ contains a vertex with distance at most $\frac{2\mu_M}{5}$ from all other points of $T.$ 
\end{proposition}

\begin{proof}
Fix $\delta < \frac{2 \mu_M}{5}.$  Geodesic balls on $(\Sigma,g)$ of radius $5\delta/2$ are embedded.  
Let $\mathcal{P}$ be a maximal subset of $(\Sigma, g_0)$ for which the pairwise distances between all points of $\mathcal{P}$ are at least $\delta.$  For $p \in \mathcal{P},$  define the Dirichlet polygons  
$$D_p := \{ x \in \Sigma: \mathrm{dist}_g(x,p) \leq \mathrm{dist}_g(x,p') \text{ for all }  p' \in \mathcal{P}\setminus\{p\} \}.$$
The $\{D_p\}_{p\in \mathcal{P}}$ tile $\Sigma.$  For every $q \in \Sigma,$ $B_{\delta}(q)$ contains some $p' \in \mathcal{P}$ by maximality of $\mathcal{P}.$  Therefore, if $q$ lies in $D_p,$ 
$$\mathrm{dist}_g(p,q) \leq \mathrm{dist}_g(p',q) \leq \delta.$$ 

Therefore, 
\begin{equation} \label{diameterdirichletpolygon}
\mathrm{diam}_g(D_p) \leq \delta.
\end{equation}

Every edge on the boundary of $D_p$ is a segment $F_{p,p'}$ of some bisector $C_{p,p'} = \{ x \in \Sigma: \mathrm{dist}_{g_0}(x,p) = \mathrm{dist}_{g_0}(x,p') \}$ for $p'\in \mathcal{P}.$  Therefore, if $q$ lies on $F_{p,p'},$
$$\mathrm{dist}_g(p,p') \leq \mathrm{dist}_g(p,q) + \mathrm{dist}_g(q,p') = 2 \mathrm{dist}(p,q) \leq 2 \delta.$$

Thus, the number of faces of $D_p$ is at most the number of $p' \in \mathcal{P}$ contained in $B_{2\delta}(p).$ For   $F_{p,p'}$ and $F_{p,p''}$ nonempty,  $B_{\frac{\delta}{2}}(p')$ and $B_{\frac{\delta}{2}}(p'')$   are disjoint subsets of $B_{\frac{5\delta}{2}}(p).$  Therefore, there are at most 
$\frac{\vol(B_{5\delta/2}(p))}{V}$ such points $p',$ where $V:=\inf_{q \in \Sigma} \vol (B_{\delta/2}(q)).$  Because the curvature $K_g \in [-C,-1],$ 
\begin{equation} \label{numberfacesdirichletpolygon}
\# \text{ faces of } D_p \leq \frac{V_{5\delta/2}(-C)}{V_{\delta/2}(-1)}, 
\end{equation}  
  Joining $p$ to the vertices of $D_p$ using geodesics, we may thus triangulate $D_p$ by at most $\frac{V_{5\delta/2}(-C)}{V_{\delta/2}(-1)}$ triangles of diameter at most $\delta.$ 

Covering $  \mathcal{P}$ by disjoint balls of radius $\delta/2$ gives 
$$\# \mathcal{P} \leq \frac{\vol_g(\Sigma)}{V_{\delta/2}(-1)}.$$

Therefore, we may triangulate $(\Sigma,g)$ by at most 
$$\frac{\vol_g(\Sigma)}{V_{\delta/2}(-1)} \cdot  \frac{V_{5\delta/2}(-C)}{V_{\delta/2}(-1)}$$
triangles of diameter at most $\delta.$

\end{proof}

\begin{proposition} \label{boundingintersectionnumbers}
Suppose $\Sigma$ is a stable, properly embedded, minimal
 surface in the closed hyperbolic 3-manifold $M.$  Suppose $\Sigma$ with its induced metric $g$ has curvature $K_g \in [-C, -1].$ Let $\mu_M$ denote the injectivity radius lower bound of Lemma \ref{muM}. Let $\gamma$ be a geodesic in $M$ intersecting $\Sigma$ transversely. Let $\# \left( \Sigma \cap \gamma \right)$ denote the absolute value of the topological intersection number of $\Sigma$ and $\gamma$.  Then
$$\# \left( \Sigma \cap \gamma \right) \leq \frac{\vol_g(\Sigma)}{V_{\frac{\mu_M}{5}}(-1)} \cdot  \frac{V_{\mu_M}(-C)}{V_{\frac{\mu_M}{5}}(-1)} \cdot \frac{\ell(\gamma)}{6 \; \mu_M / 5}.$$
\end{proposition}

\begin{proof}
  By Proposition \ref{boundedgeometrysurface}, there is a triangulation of $(\Sigma,g)$ having at most  $\frac{\vol_g(\Sigma)}{V_{\frac{\mu_M}{5}}(-1)} \cdot  \frac{V_{\mu_M}(-C)}{V_{\frac{\mu_M}{5}}(-1)}$ triangles $T$, each containing a vertex $p_T$ of distance at most $\frac{2\mu_M}{5}$ from all other points of the triangle.  
Apply the ``straightening map" $\sigma_t$ \cite[\S 11.6]{Ratcliffe} to $\Sigma,$ the linear homotopy deforming every triangle $T$ from the above triangulation to a geodesic triangle in $M$ with the same vertices.  Every geodesic triangle $T' = \sigma_1(T)$ contains a vertex $p_{T'}$ of distance at most $\frac{2\mu_M}{5}$ from all other points of $T'.$  Because $\sigma_1(\Sigma)$ and $\Sigma$ are homotopic,
\begin{equation} \label{equalintersectionnumbers}
\# (\Sigma \cap \gamma) = \# (\sigma_1(\Sigma) \cap \gamma).
\end{equation}
  We bound the right side of \eqref{equalintersectionnumbers} by the cardinality of $\sigma_1(\Sigma) \cap \gamma$. 

Let $p$ be a point at which $\gamma$ intersects $T' = \sigma_1(T).$  The geodesic ball $B_{\mu_M - \epsilon}(p_{T'})$ is embedded in $M.$  The distance from $p$ to the boundary of the ball $B_{\mu_M - \epsilon}(p)$ is at least $\mu_M - \epsilon - \frac{2 \mu_M}{5} = \frac{3 \mu_M}{5} - \epsilon.$  The geodesic enters the ball, intersects $T'$ at $p,$ then exits the ball without intersecting any other points of $T'.$  Each component of $\gamma \cap B_{\mu_M - \epsilon}(p_{T'})$ which intersects $T'$ therefore has length greater than or equal to $2( \frac{3 \mu_M}{5} - \epsilon) = \frac{6\mu_M}{5} - 2 \epsilon$, and  contains only one intersection point with $T'$. Therefore,
\begin{equation} \label{intersectionsgoodcell}
\# (T' \cap \gamma) \leq \frac{\ell(\gamma)}{6 \mu_M / 5}.
\end{equation}
Summing \eqref{intersectionsgoodcell} over all cells of the cellulation proves the desired result.
\end{proof}

\begin{lemma}
Let $M$ be a closed 3-manifold.  Suppose $H_1(M,\mathbb{Z}) = \mathbb{Z} \langle \gamma_1 \rangle \oplus \cdots \oplus \mathbb{Z} \langle \gamma_k \rangle \oplus \mathrm{finite}$ and $H_2(M, \mathbb{Z}) = \mathbb{Z} \langle S_1 \rangle \oplus \cdots \oplus \mathbb{Z} \langle S_k \rangle \oplus \mathrm{finite}.$  Then some multiple of $\gamma - \sum_{i = 1}^k n_i(\gamma) \gamma_i$ bounds, where
$$n_r = \frac{\det(A_r)}{\det(S_i \cap \gamma_j)} = \frac{\det(A_r)}{\pm 1},$$
where $A_r$ is the intersection matrix $A = (S_i \cap \gamma_j)$ with $r$th column replaced by the column $(S_i \cap \gamma).$ 
\end{lemma}

\begin{proof}
This follows by Cramer's rule, upon observing that $\vec{x}  =  \left(\begin{array}{c}n_1(\gamma)\\\vdots\\n_k(\gamma)\end{array}\right) $   is the unique solution to the system of equations
$$Ax = ( S_i \cap \gamma).$$
Also, $\det(A) = \pm 1$ because $A$ is an invertible integer matrix. 
\end{proof}

\begin{theorem} \label{lengthfreepart}
Let $M_0$ be a closed hyperbolic 3-manifold.  Let $M \rightarrow M_0$ be a finite cover.  Let $\gamma \in \Gamma = \pi_1(M).$  Let $H_2(M,\mathbb{Z}) = \mathbb{Z} \langle S_1 \rangle \oplus \cdots \oplus \mathbb{Z} \langle S_k \rangle \oplus \mathrm{finite}.$  Suppose that the minimal area surface $\Sigma_i$ in the homology class of $S_i$ is at most $A$ for $i = 1,\ldots,k.$  Let 

$$D := \begin{cases}  \mathrm{diam}(F_0) \cdot \left(  \frac{4k_0}{r_0} \cdot \mathrm{diam}(M) + 4k_0 + 1 \right)
&\text{ if } F \text{ is tree-type} \\ 2 \cdot \mathrm{diam}(M) &\text{ if } F \text{ is Dirichlet}.\end{cases}$$   

Then there exists $\gamma_0 \in \Gamma$ satisfying 
$$\ell(\gamma_0) \leq \left(A \cdot D \cdot \sqrt{b_1(M)} \right)^{b_1(M)} \cdot D \cdot b_1(M)$$
for which some multiple of $\gamma - \gamma_0$ bounds.
\end{theorem}

\begin{proof}
We can find a basis $\gamma_i, i = 1,\ldots, k,$ for $H_1(M,\mathbb{Z}) / \mathrm{torsion}$ with $\gamma_i \in P_F,$ where $F$ is either a tree-type or Dirichlet fundamental domain.  By Propositions \ref{diameterboundtreetypefundamentaldomain} and \ref{upperbounddiameterdirichletdomain},   
$$\ell(\gamma_i) \leq D \text{ for } i = 1,\ldots,k.$$
Use Proposition \ref{boundingintersectionnumbers} to bound the entries of the matrix $A_r$ obtained by replacing the $r$th column of $A = (S_i \cap \gamma_j)$ with $(S_i \cap \gamma).$  The determinant is bounded above by the products of the norms of its columns.  The result follows. 
\end{proof}

%
%

\subsubsection{Regulator-independent upper bound for $\frac{1}{\lambda_1^1(M)_{d^\ast}}$ when $b_1(M) = 1$} \label{regulatorindependentlambda1bound}

The next Proposition proves a ``regulator-independent" upper bound for $\frac{1}{\lambda_1^1(M)_{d^\ast}}$ for closed hyperbolic 3-manifolds $M$ satisfying $b_1(M) = 1$ in terms of the stable area of loops \emph{which project trivially to} $H_1(M,\mathbb{Q}).$ 

Let $\gamma\to [\gamma]$ denote the quotient map from $\pi_1(M)$ (or even closed curves) to $H_1(M,\IQ).$ 
\begin{proposition} \label{regulatorindependentbounds}
Let $M_0$ be a closed hyperbolic 3-manifold.  Let $M \rightarrow M_0$ be a finite cover.  Suppose $H_1(M,\mathbb{Z}) = \mathbb{Z} \langle \gamma_0 \rangle \oplus \mathrm{finite},$ where $\gamma_0\in P_F$,   for $F$ either a tree-type or Dirichlet fundamental domain of $M.$  Let $H_2(M,\mathbb{Z}) = \mathbb{Z} \langle S_0 \rangle \oplus \mathrm{finite}.$  Fix a basepoint $q_0 \in M.$  

Let $\lambda = \lambda_1^1(M)_{d^\ast}.$  Fix $\delta > 0.$  There is an explicit subset $\mathcal{C}_f \subset \pi_1(M)$ which projects trivially to $H_1(M,\mathbb{Q})$ relative to which the following upper bound on $\frac{1}{\lambda}$ holds: 

\begin{align} \label{regulatorindependentboundsequation}
 \frac{1}{\sqrt{\lambda}} \cdot (1 - E) &\leq    3\pi\cdot C( \lambda)^2 \cdot  V +C( \lambda) \cdot \frac{1}{2} \cdot \sqrt{\vol(M)} \nonumber \\ 
&+ C(  \lambda)^2 \cdot \left( 2\sqrt{2} D^2 \cdot \vol(M)^{\delta + 1/2} \cdot \sup_{\gamma \in \mathcal{C}_f} \frac{\mathrm{sArea}(\alpha_{\gamma q_0,q_0})}{\ell( \gamma )} +  5\pi  \right)   . \nonumber\\
\end{align}

In \eqref{regulatorindependentboundsequation}, $V$ and $D$ are the quantities from Theorem \ref{bodyupperboundpositivebettinumber}, and they satisfy the upper bounds therein.  
 The number $E$ satisfies $0 \leq E \leq C(\lambda) \cdot \vol(M)^{-\delta}.$
\end{proposition}

\begin{proof}
Let $f$ be coexact and satisfy $\Delta f = \lambda f.$ 
If 
$$\left | \int_{\alpha_{q_j,\gamma q_j}} f \right| \leq \vol(M)^{-\delta} ||f||_2$$
for every $\gamma \in P_F$ which is non-trivial in $H_1(M,\mathbb{Z}) / \mathrm{torsion},$ then we subtract this contribution from both sides of \eqref{lowerboundeigenvalueequation}.  Then \eqref{regulatorindependentboundsequation} follows directly from Corollary \ref{lowerboundeigenvalue}, with $\mathcal{C}_f=\{\gamma \in \Gamma:[\gamma]=0\text{ and }  \ell(\gamma) \leq D\}.$  

Otherwise, choose $\gamma_{\mathrm{big \; period}} \in P_F$ from among those $\gamma$ occuring on the right side of \eqref{lowerboundeigenvalueequation} which is non-zero in $H_1(M,\mathbb{Q})$ and which satisfies  
\begin{equation} \label{bigperiod}
\left | \int_{\alpha_{q_0, \gamma_{\mathrm{big \; period}} q_0}} f \right| \geq \vol(M)^{-\delta} ||f||_2
\end{equation}     
Let $h$ be a harmonic 1-form satisfying 

\begin{equation} \label{harmonic1formmatchingperiod}
\int_{\alpha_{q_0,\gamma_{\mathrm{big \; period}} q_0}} f = \int_{\gamma_{\mathrm{big \; period}}} h.
\end{equation}

Let $[\gamma_{\mathrm{big\; period}}] = m [\gamma_0]$.
 Let $A = \frac{1}{\mathrm{reg}_1(M)}.$  By definition of the regulator on 1-forms,
\begin{align} \label{matchingharmonic1formbigperiod}
||h||_2 &= \left|\int_{\gamma_0} h \right| A \nonumber \\
&= \frac{1}{m} \left| \int_{\gamma_{\mathrm{big \; period}}} h \right| \cdot A \nonumber \\
&\geq \frac{A}{m} \cdot \vol(M)^{-\delta} \cdot ||f||_2. \nonumber \\
\end{align}

Let $\gamma\in P_F$ be non-zero in $H_1(M,\mathbb{Q}).$  Let $n_0 = n_0(\gamma) = S_0 \cap \gamma$ be the unique integer for which $\gamma - n_0 \gamma_0 = 0 \in H_1(M,\mathbb{Q}).$  Then $m \gamma - n_0 \cdot \gamma_{\mathrm{big \; period}} = 0 \in H_1(M,\mathbb{Q})$ too.  Define the following three broken geodesic polygons:
\begin{itemize}
\item[($\Delta$)]
The broken geodesic $\alpha_{q_0, \gamma^m q_0} + \alpha_{\gamma^m q_0, \gamma_{\mathrm{big \; period}}^{-n_0} \cdot \gamma^m q_0} + \alpha_{\gamma_{\mathrm{big \; period}}^{-n_0} \cdot \gamma^m q_0,q_0},$ is the oriented boundary of a geodesic triangle $\Delta$ of area at most $\pi.$   

\item[($Q$)]
The broken geodesic $\alpha_{\gamma^m q_0,\gamma_{\mathrm{big \; period}}^{-n_0} \gamma^m q_0} + \alpha_{\gamma_{\mathrm{big \; period}}^{-n_0} \gamma^m q_0, \gamma_{\mathrm{big \; period}}^{-n_0} q_0} + \alpha_{\gamma_{\mathrm{big \; period}}^{-n_0} q_0,q_0} + \alpha_{q_0,\gamma^m q_0}$ forms the oriented boundary of a broken geodesic quadrilateral $Q,$ of area at most $2\pi,$ whose projection to $M$ has boundary $\alpha_{\gamma^m q_0,\gamma_{\mathrm{big \; period}}^{-n_0} \gamma^m q_0} + \alpha_{\gamma_{\mathrm{big \; period}}^{-n_0} q_0,q_0}.$  

\item[($Q'$)]
Let $\gamma = \gamma_j$ and let $q_j \in \Sigma_j$ be the reference point from Corollary \ref{lowerboundeigenvalue}.  The broken geodesic $\alpha_{q_j, \gamma q_j} + \alpha_{\gamma q_j, \gamma q_0} + \alpha_{\gamma q_0, q_0} + \alpha_{q_0, q_j}$ forms the oriented boundary of a broken geodesic quadrilateral $Q'$ of area at most $2\pi,$ whose projection to $M$ has boundary $\alpha_{q_j, \gamma q_j} - \alpha_{q_0,\gamma q_0}.$   
\end{itemize}

By three applications of Stokes,  

\begin{align} \label{periodupperboundpositivebettinumbersarea}
&\left| \int_{\alpha_{q_j, \gamma q_j}} (f-h)  \right| \nonumber \\
&\leq \left|\int_{\alpha_{q_0,\gamma q_0}} (f - h) \right| + \left| \int_{Q'} d(f-h) \right| \nonumber\\
&= \frac{1}{m} \left| \int_{\alpha_{q_0, \gamma^m q_0}} (f-h) \right| +  \left| \int_{Q'} d(f-h) \right| \nonumber \\
&\leq  \frac{1}{m} \left( \left| \int_{\alpha_{\gamma_{\mathrm{big \; period}}^{-n_0} \cdot \gamma^m q_0,q_0}} (f-h) \right| +  \left| \int_{\alpha_{\gamma^m q_0, \gamma_{\mathrm{big \; period}}^{-n_0} \gamma^m q_0}} (f-h) \right| + \left| \int_{\Delta} d(f-h)\right| \right) +  \left| \int_{Q'} d(f-h) \right| \nonumber  \\
& \leq  \frac{1}{m} \left( \left| \int_{\alpha_{\gamma_{\mathrm{big \; period}}^{-n_0} \cdot \gamma^m q_0,q_0}} (f-h) \right| +  \left| \int_{\alpha_{q_0, \gamma_{\mathrm{big \; period}}^{-n_0} q_0}} (f-h) \right| + \left| \int_Q d(f-h)  \right| + \left| \int_{\Delta} d(f-h)\right| \right) +  \left| \int_{Q'} d(f-h) \right| \nonumber \\
& \leq  \frac{1}{m} \left( \left| \int_{\alpha_{\gamma_{\mathrm{big \; period}}^{-n_0} \cdot \gamma^m q_0,q_0}} (f-h) \right| + \left| n_0 \cdot \int_{\alpha_{q_0, \gamma_{\mathrm{big \; period}} q_0}} (f-h) \right| \right) + ||df||_{\infty} \cdot \left( \frac{\mathrm{area}(Q)}{m}  +  \frac{\mathrm{area}(\Delta)}{m} + \mathrm{area(Q')} \right) \nonumber \\
&=  \frac{1}{m}   \left| \int_{\alpha_{\gamma_{\mathrm{big \; period}}^{-n_0} \cdot \gamma^m q_0,q_0}} (f-h) \right|     + ||df||_{\infty} \cdot \left( \frac{\mathrm{area}(Q)}{m} + \frac{\mathrm{area}(\Delta)}{m} + \mathrm{area}(Q') \right) \nonumber \\
&\leq  \frac{1}{m} \left| \int_{\alpha_{\gamma_{\mathrm{big \; period}}^{-n_0} \cdot \gamma^m q_0,q_0}} (f-h) \right| +  ||df||_{\infty} \cdot 5\pi  \nonumber \\
&\leq \frac{1}{m} ||df||_{\infty} \cdot \mathrm{sArea}(\alpha_{\gamma_{\mathrm{big \; period}}^{-n_0} \cdot \gamma^m q_0,q_0})  +  ||df||_{\infty} \cdot 5\pi \nonumber \\
&\leq \sqrt{\lambda} \cdot ||f||_2 \cdot C(  \lambda) \cdot \left( \frac{\ell( \gamma_{\mathrm{big \; period}}^{-n_0(\gamma)} \cdot \gamma^m)}{m} \cdot \frac{\mathrm{sArea}(\alpha_{\gamma_{\mathrm{big \; period}}^{-n_0} \cdot \gamma^m q_0,q_0})}{\ell( \gamma_{\mathrm{big \; period}}^{-n_0(\gamma)} \cdot \gamma^m)} + 5\pi  \right). \nonumber \\
\end{align}


By Proposition \ref{boundingintersectionnumbers},
\begin{align} \label{upperboundtranslationlengthintersectionnumbers}
\frac{\ell( \gamma_{\mathrm{big \; period}}^{-n_0(\gamma)} \cdot \gamma^m)}{m} &\leq \frac{1}{m} \cdot \left( m \cdot \ell(\gamma) + |n_0(\gamma)| \cdot \ell(\gamma_{\mathrm{big \; period}})  \right) \nonumber \\
&=  \ell(\gamma) + \frac{\# \left( [S_0] \cap \gamma \right)}{m} \cdot \ell(\gamma_{\mathrm{big \; period}}) \nonumber \\
&\leq \ell(\gamma) + \frac{ \inf_S \left( \mathrm{area}(S) \right) \cdot \ell(\gamma)V_{\frac{\mu_M}{5}}(-C)}{m V_{\frac{\mu_M}{5}}(-1)^2\mu_M} \cdot \ell(\gamma_{\mathrm{big \; period}}), 
\end{align}
where $S$ runs over all stable, minimal surfaces representing the homology class $[S_0].$ 


Let $\Sigma$ be  a stable minimal surface representing $[S_0]$. 
Let $\gamma_0^\vee \in H^1(M,\mathbb{R})$ be Poincar\'{e} dual to $[S_0]$; $\gamma_0^\vee$ may be viewed as a linear map $H_1(M,\mathbb{R}) \rightarrow \mathbb{R}$ and is uniquely determined by the condition $\gamma_0^\vee(\gamma_0) = 1.$  Then 
\begin{align} \label{leastareavsharmonic}
\mathrm{area}(\Sigma) &=     \inf \{ || \alpha ||_{L^1(M)}: \alpha \in \Omega^1(M) \text{ represents } \gamma_0^\vee \}, \text{ by \cite[Lemma 3.1]{BD2}} \nonumber \\
&\leq    || h_0 ||_{L^1(M)}, \text{ for } h_0 = \text{ harmonic representative of } \gamma_0^\vee \nonumber \\
&\leq     || h_0 ||_{L^2(M)} \cdot \sqrt{\vol(M)}, \text{ by Cauchy-Schwartz} \nonumber \\
&=:  A \cdot \sqrt{\vol(M)}. 
\end{align} 
By Propositions \ref{diameterboundtreetypefundamentaldomain} and \ref{upperbounddiameterdirichletdomain},
\begin{equation} \label{shorttranslationlength}
\ell(\gamma), \ell(\gamma_{\mathrm{big \; period}}) \leq D := \begin{cases}  \mathrm{diam}(F_0) \cdot \left(  \frac{4k_0}{r_0} \cdot \mathrm{diam}(M) + 4k_0 + 1 \right)
&\text{ if } F \text{ is tree-type} \\ 2 \cdot \mathrm{diam}(M) &\text{ if } F \text{ is Dirichlet}.\end{cases}
\end{equation}
Inserting \eqref{leastareavsharmonic} and \eqref{shorttranslationlength} into \eqref{upperboundtranslationlengthintersectionnumbers} yields
\begin{align} \label{finalupperboundtranslationlength}
\frac{\ell( \gamma_{\mathrm{big \; period}}^{-n_0(\gamma)} \cdot \gamma^m)}{m} 
&\leq  \ell(\gamma) + \frac{ A \cdot \sqrt{\vol(M)}V_{\frac{\mu_M}{5}}(-C)  \cdot \ell(\gamma)}{mV_{\frac{\mu_M}{5}}(-1)^2\mu_M} \cdot \ell(\gamma_{\mathrm{big \; period}}) \nonumber\\
&\leq D^2 \cdot \left( 1 + \frac{ A \cdot \sqrt{\vol(M)}V_{\frac{\mu_M}{5}}(-C)}{mV_{\frac{\mu_M}{5}}(-1)^2\mu_M} \right).  
\end{align}

Combining \eqref{periodupperboundpositivebettinumbersarea}, \eqref{finalupperboundtranslationlength}, and \eqref{matchingharmonic1formbigperiod} implies that

\begin{align*}
&\frac{1}{ ||f - h||_2} \cdot \left|\int_{\alpha_{q_j,\gamma q_j}} (f - h) \right| \\
&= \frac{1}{(||f||_2^2 + ||h||_2^2)^{1/2}}  \cdot \left|\int_{\alpha_{q_j,\gamma q_j}} (f - h) \right| \\
&\leq \sqrt{\lambda} \cdot C(  \lambda) \cdot \frac{||f||_2}{(||f||_2^2 + ||h||_2^2)^{1/2}} \cdot \left( \frac{\ell( \gamma_{\mathrm{big \; period}}^{-n_0(\gamma)} \cdot \gamma^m)}{m} \cdot \frac{\mathrm{sArea}(\alpha_{\gamma_{\mathrm{big \; period}}^{-n_0} \cdot \gamma^m q_0,q_0})}{\ell( \gamma_{\mathrm{big \; period}}^{-n_0(\gamma)} \cdot \gamma^m)} + 5\pi  \right) \\
&\leq \sqrt{\lambda} \cdot C(  \lambda) \cdot \left(  \left( \frac{||f||_2}{(||f||_2^2 + ||h||_2^2)^{1/2}} \cdot \frac{\ell( \gamma_{\mathrm{big \; period}}^{-n_0(\gamma)} \cdot \gamma^m)}{m} \right) \cdot \frac{\mathrm{sArea}(\alpha_{\gamma_{\mathrm{big \; period}}^{-n_0} \cdot \gamma^m q_0,q_0})}{\ell( \gamma_{\mathrm{big \; period}}^{-n_0(\gamma)} \cdot \gamma^m)} + 1 \cdot 5\pi  \right) \\
&\leq \sqrt{\lambda} \cdot C( \lambda) \cdot \left(  \left( \frac{||f||_2}{(||f||_2^2 + ||h||_2^2)^{1/2}} \cdot  D^2 \cdot \left( 1 + \frac{2A \cdot \sqrt{\vol(M)}}{m} \right) \right) \cdot \frac{\mathrm{sArea}(\alpha_{\gamma_{\mathrm{big \; period}}^{-n_0} \cdot \gamma^m q_0,q_0})}{\ell( \gamma_{\mathrm{big \; period}}^{-n_0(\gamma)} \cdot \gamma^m)} + 5\pi  \right) \\
&\leq \sqrt{\lambda} \cdot C(  \lambda) \cdot \left(  \left( \sqrt{2} \cdot \frac{||f||_2}{||f||_2 + ||h||_2} \cdot  D^2 \cdot \left( 1 + \frac{2A \cdot \sqrt{\vol(M)}}{m} \right) \right) \cdot \frac{\mathrm{sArea}(\alpha_{\gamma_{\mathrm{big \; period}}^{-n_0} \cdot \gamma^m q_0,q_0})}{\ell( \gamma_{\mathrm{big \; period}}^{-n_0(\gamma)} \cdot \gamma^m)} + 5\pi  \right) \\
&\leq \sqrt{\lambda} \cdot C(  \lambda) \cdot \left(  \left(  \frac{\sqrt{2} D^2 \cdot \left( 1 + \frac{2A \cdot \sqrt{\vol(M)}}{m} \right)}{1 + \frac{A \cdot \vol(M)^{-\delta}}{m}} \right) \cdot \frac{\mathrm{sArea}(\alpha_{\gamma_{\mathrm{big \; period}}^{-n_0} \cdot \gamma^m q_0,q_0})}{\ell( \gamma_{\mathrm{big \; period}}^{-n_0(\gamma)} \cdot \gamma^m)} +  5\pi  \right) \\
&\leq \sqrt{\lambda} \cdot C(  \lambda) \cdot \left( 2\sqrt{2} D^2 \cdot \vol(M)^{\delta + 1/2} \cdot \frac{\mathrm{sArea}(\alpha_{\gamma_{\mathrm{big \; period}}^{-n_0} \cdot \gamma^m q_0,q_0})}{\ell( \gamma_{\mathrm{big \; period}}^{-n_0(\gamma)} \cdot \gamma^m)} + 5\pi  \right). \\
\end{align*}

Using this final inequality to bound the period integrals from Corollary \ref{lowerboundeigenvalue}, the Proposition follows for 
\begin{align*}
\mathcal{C}_f &= \{\gamma \in \Gamma: \ell(\gamma) \leq D, \gamma \text{ projects trivially to } H_1(M,\mathbb{Q})  \} \\
&\bigcup \{ \gamma_{\mathrm{big \; period}}^{-n_0(\gamma)} \cdot \gamma^{m(\gamma)}: \ell(\gamma) \leq D, \gamma \text{ projects non-trivially to } H_1(M,\mathbb{Q})   \}.
\end{align*}

\end{proof}

\section{Comparing combinatorial and Riemannian $L^p$-norms on the Whitney complex} \label{whitneycomplex}
\subsubsection{Notation and setup}
Let $M_0$ be a closed hyperbolic $n$-manifold and $M \xrightarrow{\pi} M_0$ an arbitrary finite cover.  Let $K_0$ be a triangulation of $M_0.$ We define an integer valued distance function on the $n$-simplices by defining for $\sigma\not = \tau$, $d(\sigma,\tau) = 1$ if $\sigma\cap\tau\not = \emptyset.$  The triangle inequality yields a unique minimal integer valued extension. For every top degree simplex  $\sigma,$ let $B_r(\sigma)$ denote those simplices at distance at most $r$ from $\sigma$.  Assume that $K_0$ is fine enough so that $B_2(\sigma)$ is contained in an embedded geodesic ball for every top degree simplex $\sigma.$  

 Let $C^q(M;K)$ denote the space of real-valued cochains of the triangulation $K.$  Let $C_q(M; K)$ denote the space of real chains.  We denote cochains by Greek letters $\alpha, \beta,\ldots$ and chains by Roman letters $a,b,\ldots$  

Let $K = \pi^{-1}(K_0)$ - the pulled-back triangulation of $M.$  For every $q$-cell $c \in K,$ let $\mathbf{1}_c$ denote the dual cochain.  Let $W$ denote the Whitney map \cite{RS} 
$$W: C^q(M; K) \rightarrow \Omega^q(M).$$
By construction, for each cell $c$,
\begin{align}\int_cW(\mathbf{1}_c) = 1.
\end{align}
Hence, the Whitney map respects the duality between chains and cochains. 

\subsubsection{Norms on cochains}
The cochain spaces $C^q(M;K)$ have two natural families of norms: 
\begin{align*}
||\gamma||_{p,M} &:= || W(\gamma) ||_{p,M} \\
\left|\left| \sum_i a_i \mathbf{1}_{c_i} \right| \right|_{p, \mathrm{comb}} &:= \left( \sum_i |a_i|^p \right)^{1/p}
\end{align*}
Above, $|| \cdot ||_{p,M}$ denotes the Riemannian $L^p$ norm on $\Omega^q(M).$ 

Because all norms are equivalent on finite dimensional vector spaces, for $1\leq p,m\leq \infty$, there exists  $A_{p,m,M_0}$ so that 
\begin{align}\label{pmcompare}A^{-1}_{p,m,M_0}\|\gamma\|_{m,M_0}\leq \|\gamma\|_{p,\mathrm{comb}}\leq A_{p,m,M_0}\|\gamma\|_{m,M_0} \,\,\forall \gamma \in C^q(M_0;K_0).
\end{align}



\subsubsection{Norms on chains} \label{chainnorm}
Let $V,V'$ be two finite dimensional real vector spaces together with a perfect bilinear pairing $B: V \times V' \rightarrow \mathbb{R}.$  Suppose $V$ is equipped with norm $|| \cdot ||.$  We can define a norm on $V'$ by
\begin{align}\label{normdef}||w||^{B} := \sup_{0 \neq v \in V} \frac{\left| B(v,w) \right|}{||v||}.\end{align}
With this choice of norm
\begin{align*}
V' &\rightarrow V^\vee \\
w &\mapsto B(\cdot,w)
\end{align*}
is isometric, where $V^\vee$ denotes the dual space of $V$ equipped with its usual dual norm.  This procedure for defining a norm is \emph{reflexive} in the sense that  
$$||v|| = \sup_{0 \neq w \in V'} \frac{\left|B(v,w) \right|}{||w||^{B}} \text{ for all } v \in V.$$
With $V=C^q(M;K) $, $V'=C_q(M;K)$, and $B$ the canonical evaluation pairing, we denote by $\|\cdot\|_{p',\mathrm{comb}}$  (respectively $\|\cdot\|_{p',M}$) the norms induced on $C_q(M;K)$ by the norms $\|\cdot\|_{p,\mathrm{comb}}$  (respectively $\|\cdot\|_{p,M}$)  on $C^q(M;K),$ with $\frac{1}{p}+\frac{1}{p'} = 1.$ For the combinatorial norm, this gives the expcted 
$$\left|\left|\sum_i a_i c_i \right|\right|_{p',\mathrm{comb}} := \left(\sum |a_i|^{p'} \right)^{1/p'},$$
but the norms dual to the Whitney norms do not admit such a simple expression. 

\subsection{Comparing the Riemannian $L^2$ and $L^\infty$ norms on the Whitney complex}
\begin{define} The combinatorial support of a cochain $c= \sum_i a_i \mathbf{1}_{c_i}$ is $\cup_{a_i\not = 0}c_i.$
\end{define}
The Whitney map satisfies the property that if $c_j$ is a $q$ cell contained in the $n$ cell $\sigma$, then the support of $W(\mathbf{1}_{c_j})$ is contained in $B_1(\sigma)$. The following lemma shows that combinatorial localization is bounded. 
 
\begin{lemma} \label{operatornormnonorthogonalprojection}
Let $\sigma$ be a top degree cell of $K.$  Define the projection 
\begin{align*}
P_{\mathrm{comb}} : ( C^q(M;K), ||\cdot||_{2,M}) &\rightarrow (C^q(M;K), ||\cdot||_{2,M}) \\
\sum_i x_i \mathbf{1}_{c_i} &\mapsto \sum_{i: c_i \cap \sigma\not =\emptyset} x_i \mathbf{1}_{c_i},
\end{align*}
which is orthogonal with respect to the combinatorial-$L^2$ inner product. There exists $D_{M_0}>0$ (depending only on $M_0$) such that $\|P_{\mathrm{comb}}\| <D_{M_0}$.  
\end{lemma}

\begin{proof}
\begin{align} \label{operatornormnonorthogonalprojection1}
||P_{\text{comb}}|| &= \sup_{0 \neq z } \frac{|| W(P_{\text{comb}}(z)) ||_{2,M}}{||W(z)||_{2,M }}.
\end{align}
Write $z = z_1+z_2+z_3$, where $z_1=P_{\text{comb}}(z)$, 
$z_2$ has combinatorial support in $B_2(\sigma)\setminus B_1(\sigma)$, and $z_3$ has combinatorial support in $B_2(\sigma)^c$.  Then 
$$\|W(z)\|^2_{2,M}=\|W(z_1)+ \chi_{B_2(\sigma)} W(z_2)\|^2_{2,M}+\|W(z_3)+(1- \chi_{B_2(\sigma)} )W(z_2)\|^2_{2,M},$$
where $\chi_A$ denotes the characteristic function of $A$.
 Hence 
$$\|W(z)\|^2_{2,M}\geq \inf \{ \|W(z_1)+ \chi_{B_2(\sigma)}W(z_2)\|^2_{2,M} :z_2\text{ has combinatorial support in }B_2(\sigma)\setminus B_1(\sigma) \}.$$
The right hand side of the preceding inequality is nonzero for $z_1\not =0$  because $W(z_1)$ and $\chi_{B_2(\sigma)}W(z_2)$ are linearly independent.  Hence, it defines a new norm $\|\cdot\|_{\mathrm{quot}}$ (and associated inner product structure) on the cochains with combinatorial support in $B_1(\sigma)$. Since the vector space is finite dimensional, there exists $C>0$ such that 
$$\|W(z_1)\|_{2,M}\leq C\|z_1\|_{\mathrm{quot}}.$$ 
The assumption that $B_2(\sigma)$ is contained in an embedded geodesic ball implies that $\frac{1}{C}=: D_{M_0}$ depends only on $K_0$.  Thus, \begin{align} \label{operatornormnonorthogonalprojection2}
||P_{\text{comb}}|| &\leq  \sup_{0 \neq P_{\text{comb}}(z) } \frac{|| W(P_{\text{comb}}(z)) ||_{2,M }}{||P_{\text{comb}}(z)||_{\mathrm{quot} }} \leq   D_{M_{0}}.
\end{align} 
\end{proof}

\begin{proposition} \label{comparisonswithwhitneysupnorm}
There exists $c_{M_0}, C_{M_0}>0$ such that for all cochains $\gamma$,
\begin{align}\label{boundingsupnormbylinftynorm}\|\gamma\|_{\infty,M}\leq c_{M_0}\| \gamma \|_{\infty,\mathrm{comb}}.
\end{align}
\begin{align}\label{boundingsupnormbyl2normofprojection}\|\gamma\|_{\infty,M}\leq C_{M_0}\|\gamma\|_{2,M}.
\end{align}
\end{proposition}
\begin{proof}
Let $\gamma = \sum_i a_i \mathbf{1}_{c_i} \in C^q(M; K)$ be a cochain.  Suppose that $|W(\gamma)|$ attains its sup at $p \in s$, for some $n$-simplex $s.$  
Then
\begin{equation*}
||\gamma||_{\infty,M} = \left|\left|\sum_{i: c_i\cap s \not = \emptyset} a_i W(\mathbf{1}_{c_i}) \right|\right|_{\infty,M}.
\end{equation*}   
Hence 
\begin{equation}\label{inftycomb}
||\gamma||_{\infty,M} \leq  c_{M_0}\|\gamma\|_{\infty,\text{comb}},
\end{equation}
with 
\begin{equation}\label{cm0}c_{M_0}:=\sup_s \left| \left| \sum_{i: c_i\cap s \not = \emptyset}   W(\mathbf{1}_{c_i}) \right| \right|_{\infty,M}.
\end{equation}
By our assumption that $B_2(\pi(s))$ is embedded,
\begin{align} 
\left| \left|\sum_{i:   c_i\cap s \not = \emptyset} a_i W(\mathbf{1}_{c_i}) \right| \right|_{\infty,M} &= \left|\left|\sum_{i:  \pi(c_i)\cap\pi(s)\not = \emptyset} a_i W(\mathbf{1}_{\pi(c_i)}) \right| \right|_{\infty,M_0} \nonumber \\
&\leq    A_{\infty,2,M_0} \left| \left|\sum_{i: \pi(c_i)\cap\pi(s)\not = \emptyset} a_i W(\mathbf{1}_{\pi(c_i)}) \right| \right|_{2,M_0}  \nonumber \\
&=   A_{\infty,2,M_0} \left| \left|\sum_{i: c_i\cap s \not = \emptyset} a_i W(\mathbf{1}_{c_i}) \right| \right|_{2,M}\\
&\leq D_{M_0} A_{\infty,2,M_0} ||\gamma||_{2,M}. 
\end{align}
   Here we have used equation \eqref{pmcompare}  and Lemma \ref{operatornormnonorthogonalprojection}.
\end{proof}


\subsection{Comparing the Riemannian and combinatorial $L^p$-norms on the Whitney complex}
\begin{proposition} \label{whitneycombinatorialequivalence}
There are inequalities
\begin{align}\label{whitneysupnormcomparison}
   A_{\infty,\infty,M_0}^{-1}\| \gamma \|_{\infty,\mathrm{comb}}\leq ||\gamma ||_{\infty,M}\leq c_{M_0}\| \gamma \|_{\infty,\mathrm{comb}} \text{ for all } \gamma \in C^q(M;K) 
\end{align}
and 
\begin{align}\label{whitneysupnormcomparison}
    c_{M_0}^{-1}\|c\|_{1,\mathrm{comb}}\leq ||c||_{1,M}\leq A_{\infty,\infty,M_0}\|c\|_{1,\mathrm{comb}} \text{ for all } c \in C_q(M;K). 
\end{align}
Above, $c_{M_0}$ denotes the constant from Proposition \ref{comparisonswithwhitneysupnorm}.  
\end{proposition}
\begin{proof}
Let $\gamma = \sum a_i \mathbf{1}_{c_i}\in C^q(M,K).$  Arguing as in the previous subsection:
\begin{align*} 
||\gamma||_{\infty,M} &= \max_{\text{top degree cells } s} \left| \left|\sum_{i:  c_i\cap s \not = \emptyset} a_i W(\mathbf{1}_{c_i}) \right|\right|_{\infty,M} \\
&= \max_{\text{top degree cells } s} \left|\left|\sum_{i:  \pi(c_i)\cap \pi(s) \not = \emptyset} a_i W(\mathbf{1}_{\pi(c_i)}) \right| \right|_{\infty,M_0} \\
&\geq A_{\infty,\infty,M_0}^{-1}\max_{\text{top degree cells } s} \max_{i:  \pi(c_i)\cap \pi(s) \not = \emptyset} |a_i| \\
&= A_{\infty,\infty,M_0}^{-1}\|\gamma\|_{\infty,\text{comb}} . 
\end{align*}
Hence
\begin{align}\label{whitneysupnormcomparison}
   A_{\infty,\infty,M_0}^{-1}\|\gamma\|_{\infty,\text{comb}}\leq ||\gamma||_{\infty,M}\leq c_{M_0}\|\gamma\|_{\infty,\text{comb}} . 
\end{align} 
The statement relating combinatorial and de Rham $L^1$-norms on chains follows by duality.   
\end{proof}



\section{The Whitney 2-chain Laplacian spectral gap controls stable commutator length} \label{whitneylaplacianscl}

Fix a triangulation $K_0$ of a closed hyperbolic $n$-manifold $M_0.$ 
Let $M \xrightarrow{p} M_0$ be a finite cover of $M_0$ with triangulation $K = p^{-1}(K_0).$  \bigskip

\begin{proposition} \label{whitneylambda1controlsscl}
Let $f \in C_1(M;K)$ be an integeral 1-chain, some multiple of which bounds.  Let $\mathrm{length}(f)$ denote the Riemannian length of $f.$\footnote{If $f = \sum a_i c_i,$ its \emph{Riemannian length} is defined to equal $\sum |a_i| \mathrm{length}(c_i).$}  Then there exists $B_{M_0}>0$, depending only on $M_0,$ an integer $m,$ and a surface $S_m \in C_2(M;K)$ bounding $nf$ satisfying  
$$\left( \frac{|\chi(S_m)| / m }{\mathrm{length}(f)} \right)^2 \leq   \frac{B_{M_0}\vol(M)^2}{\lambda_{1,\mathrm{Whitney}}^1(M)_{d^\ast}}.$$ 
\end{proposition}

\begin{proof}
Equip the chain groups $C_q(M;K)$ with the norm $|| \cdot ||_{2,M}$ dual to the norm $||\cdot ||_{2,M}$ on $C^q(M;K)$ induced from the $L^2$-norm on $\Omega^q(M)$ via the Whitney embedding.  

Let $\partial$ denote the boundary map on $C_{\bullet}(M;K).$  Let $g$ be a real $2$-chain satisfying $\partial g = f$ and $g\perp \text{ker}(\p).$ Because the kernel and the image of $\partial$ have rational bases (and so the existence of a real solution to $\partial c = f$ implies the existence of a rational solution), we can find a rational 2-chain $g'$ with $\partial g' = f$ and $|| g' - g ||_{2,M}$ arbitrarily small.  So we may assume for any $\delta>0$ there exists $g'$ so that:
\begin{itemize}
\item[(1$'$)] 
$\partial g' = f$

\item[(2$'$)]  
$|| g' ||_{2,M}^2 \leq \frac{1+\delta}{\lambda_{1,\mathrm{Whitney}}^1(M)_{d^\ast}} || f ||_{2,M}^2$

\item[(3$'$)] 
$g'$ is a rational 2-chain.
\end{itemize}

Condition (2$'$) can be guaranteed because $\lambda_1^1(M)_{\mathrm{Whitney},d^\ast}$ equals the smallest eigenvalue of $\partial^\ast \partial$ acting on $2$-chains perpendicular to $\ker \partial.$  Choose an integer $m$ so that 
\begin{equation} \label{integermultiple}
mg' = \sum_{j=1}^k a_jc_{2,j}, \text{ for } a_q \in \mathbb{Z} \text{ and } \{c_{2,j}\}_j \text{ the 2-cells of } K.
\end{equation}  
The number of faces of $mg'$ equals $\sum_{i = 1}^k |a_i| = ||mg'||_{1,\mathrm{comb}}.$  Counting incident pairs $\{ \text{vertex} \in \text{face}\}$ and $\{\text{edge} \subset \text{face} \},$ we see that
\begin{align*}
3 F &\geq V, \\
3 F &\geq E,
\end{align*}
where $V,E,F$ denote the number of vertices, edges, and faces on the surface bounding $mf.$  Therefore, the absolute value of the euler characteristic $= |(V + F) - E|$ is at most $\max \{V + F, E\} \leq  4 F \leq 4 ||mg'||_{1, \mathrm{comb}} .$  Thus, letting $S_m = mg',$ 
\begin{align} \label{sclupperbound}
|\chi(S_m) |^2  &\leq 16 ||mg'||^2_{1, \mathrm{comb}} \nonumber \\
&\leq c_{M_0} \cdot  16 \cdot ||mg'||^2_{1, M} \nonumber \hspace{0.5cm} \text{ by Proposition \ref{whitneycombinatorialequivalence}}\\
&\leq 16 c_{M_0}\cdot \vol(M) \cdot ||mg'||^2_{2,M} \nonumber\\
&\leq m^2 \cdot 16 c_{M_0} \cdot \vol(M) \cdot \frac{1 + \delta}{\lambda_{1,\mathrm{Whitney}}^1(M)_{d^\ast}} || f||^2_{2,M} \hspace{0.5cm} \text{ by (2$'$)} . 
\end{align}

Also,
\begin{align} \label{whitneynormupperbound}
||f||_{2,M} &=   \sup_{\gamma \in C^1(M;K): ||\gamma||_{2,M} \leq 1} \left|  \int_f W(\gamma) \right| \nonumber \\
&\leq   \sup_{\gamma \in C^1(M;K): ||\gamma||_{2,M} \leq 1} ||\gamma ||_{\infty,M} \,\mathrm{length}(f) \nonumber \\
&\leq  C_{M_0}  \sup_{\gamma \in C^1(M;K): ||\gamma||_{2,M} \leq 1} ||\gamma||_{2,M}\, \mathrm{length}(f) \hspace{0.5cm} \text{ by Proposition \ref{comparisonswithwhitneysupnorm}}\nonumber \\
&=  C_{M_0}  \cdot \mathrm{length}(f).
\end{align}

Together \eqref{sclupperbound} and \eqref{whitneynormupperbound} yield

\begin{equation} \label{finalsclupperbound}
\left( \frac{ \chi(S_m) / m}{\mathrm{length}(f)} \right)^2 \leq B_{M_0}  \frac{\vol(M)}{\lambda_{1,\mathrm{Whitney}}^1(M)_{d^\ast}},
\end{equation}
where $B_{M_0} = C_{M_0} \cdot 16c_{M_0} \cdot (1 + \delta).$ 
\end{proof}

\begin{remark} \label{combinatoriallambda1controlsscl}
The same result is true, replacing the Whitney Laplacian with respect to $K$ and $\lambda_{1,\mathrm{Whitney}}^1(M)_{d^\ast}$ by the combinatorial Laplacian with respect to $K$ and $\lambda_{1, \mathrm{comb}}^1(M)_{d^\ast}.$  The proof actually simplifies, because comparing the combinatorial $L^1$ and $L^2$-norms is easier than comparing the combinatorial $L^1$-norm and Riemannian $L^2$-norm.
\end{remark}

\begin{theorem} \label{whitneylambda1controlsscltriangulationindependent}
Let $\gamma \in \pi_1(M)$ have translation length $\ell(\gamma).$  Suppose that some multiple of $\gamma$ bounds.  Then
$$\left(\frac{\mathrm{scl}(\gamma)}{\ell(\gamma)} \right)^2 \leq W_{M_0} \cdot \frac{\vol(M) \cdot \mathrm{diam}(M)^2}{\lambda_1^1(M)_{\mathrm{Whitney},d^\ast}}$$
for some constant $W_{M_0}$ depending only on $M_0.$
\end{theorem}

\begin{proof}
Pull back the triangulation $K$ of $M$ to a triangulation $\widetilde{K}$ on $\mathbb{H}^n.$  Let $p$ be a vertex of $\widetilde{K}$ whose distance to the minimum translation set of $\gamma$ is minimal; this choice of $p$ depends on $\gamma.$  The distance from $p$ to the minimum translation set of $\gamma$ is at most $\mathrm{diam}(M).$  

Let $\alpha$ be a geodesic segment in $\mathbb{H}^n$ from $p$ to $\gamma p.$  Suppose $s_0,\ldots, s_k$ are the top degree simplices of $\widetilde{K}$ whose interiors $\alpha$ passes through in the order listed.  We can find a path
$$\alpha_{\mathrm{comb}} = \left( c_1^0  \cdots c_{j_0}^0 \right)  \cdots \left( c_1^k \cdots c_{j_k}^k \right)$$
(dots denote concatenation, and the above expression should be read in ``left to right order")
satisfying
\begin{itemize}
\item
every $c_j^i$ is an edge of $s_i,$ 

\item
no edge is repeated (implying that $j_0,\ldots,j_k \leq {n+1 \choose 2}$),

\item
$c_1^0$ begins at $p$ and $c_{j_k}^k$ ends at $\gamma p.$ 
\end{itemize}

Let $f \in C^1(\mathbb{H}^n;\widetilde{K})$ denote the chain induced by $\alpha_{\mathrm{comb}}$; i.e.
$$f = \left( c_1^0 +  \cdots + c_{j_0}^0 \right)  + \cdots + \left( c_1^k + \cdots + c_{j_k}^k \right),$$
and let $\overline{f}$ denote its projection to $M.$  By Theorem \ref{whitneylambda1controlsscl}, there is an integer $m,$ a surface $S_m$ bounding $m \overline{f},$ and a constant $B_{M_0}$ for which
$$\left(\frac{|\chi(S_m)|}{m} \right)^2 \leq  \mathrm{length}(\overline{f})^2 \cdot \frac{B_{M_0} \vol(M)}{\lambda_1^1(M)_{\mathrm{Whitney},d^\ast}}.$$

The projection $\overline{\alpha}_{\mathrm{comb}}$ of $\alpha$  to $M$ is homotopic to $\gamma$ in $\pi_1(M).$  Therefore, 
\begin{align*}
\left(\frac{\mathrm{scl}(\gamma)}{\ell(\gamma)} \right)^2 &= \left(\frac{\mathrm{scl}(\overline{\alpha}_{\mathrm{comb}})}{\mathrm{length}(\overline{f})} \right)^2 \cdot \left(\frac{\mathrm{length}(\overline{f})}{\ell(\gamma)} \right)^2 \\
&\leq \left(\frac{| \chi(S_m) |/m}{\mathrm{length}(\overline{f})} \right)^2 \cdot \left(\frac{\mathrm{length}(\overline{f})}{\ell(\gamma)} \right)^2 \\
&\leq  \frac{B_{M_0} \vol(M)}{\lambda_1^1(M)_{\mathrm{Whitney},d^\ast}} \cdot  \left(\frac{\mathrm{length}(f)}{\ell(\gamma)} \right)^2,
\end{align*}

where passage to the last line follows from Proposition \ref{whitneylambda1controlsscl}.  
The second bullet point above implies that $\mathrm{length}(f)$ is at most ${n+1 \choose 2} \cdot k \cdot e_0,$ where $k$ is the combinatorial distance from $s_0$ to $s_k$ in the dual graph to the triangulation $\widetilde{K}$ and $e_0$ is the length of the longest edge in $K_0.$  By the argument from Lemma \ref{diametercomparison}, there are constants $a_0, b_0,$ depending only on $M_0,$ for which 
$$k \leq a_0 \cdot d(p,\gamma p) + b_0.$$
Because the distance from $p$ to the minimum translation set of $\gamma$ is at most $\mathrm{diam}(M),$ the latter inequality implies
\begin{align*}
\mathrm{length}(f) &\leq k \cdot {n+1 \choose 2} \cdot e_0 \\
&\leq \left[a_0 \cdot \left( 2 \mathrm{diam}(M) + \ell(\gamma) \right) + b_0\right] \cdot {n+1 \choose 2} \cdot e_0,
\end{align*}
where $e_0$ denotes the maximum edge length in $K_0.$  The result follows.
\end{proof}

\section{Comparing $\lambda_1^1(M)_{d^\ast}$ to $\lambda_{1,\mathrm{Whitney}}^1(M)_{d^\ast}$} \label{lambda1whitneylambda1derham}

\begin{proposition} \label{whitneyderhamcomparison}
Let $M_0$ be a closed hyperbolic manifold with triangulation $K_0.$  Let $M \xrightarrow{\pi} M_0$ be an arbitrary finite cover with pullback triangulation $K = \pi^{-1}(M_0).$  Either  
$$\lambda_{1,\mathrm{Whitney}}^1(M)_{d^\ast} \geq  \frac{1}{4G_{M_0}^2 C_{M_0}^2 \vol(M)},$$ 
or  
$$\lambda_1^1(M)_{d^\ast} \leq 4G_{M_0}^2\vol(M) \cdot \lambda_{1,\mathrm{Whitney}}^1(M)_{d^\ast},$$
where $C_{M_0}$ is defined in Proposition \ref{comparisonswithwhitneysupnorm} and $G_{M_0}$ is defined in the proof body. 
\end{proposition}

\begin{remark}
If the first alternative in Proposition \ref{whitneyderhamcomparison} holds, Theorem \ref{whitneylambda1controlsscltriangulationindependent} implies that 
$$\frac{\mathrm{scl}(\gamma)}{\ell(\gamma)} \leq \sqrt{D_{M_0}} \cdot 2 G_{M_0 }C_{M_0} \cdot \vol(M) \cdot \mathrm{diam}(M).$$ 
In particular, if $b_1(M) = 0$ or $n = 3$ and $b_1(M) = 1,$ then Corollary \ref{lowerboundeigenvalue} and Proposition \ref{regulatorindependentbounds} respectively imply that 
$$\frac{1}{\lambda_1^1(M)_{d^\ast}} \leq \begin{cases}  E_{M_0} \cdot \vol(M)^2 \cdot \mathrm{diam}(M)^4 &\text{ if } b_1(M) = 0 \\ E_{M_0,\delta} \cdot \vol(M)^{3 + 2\delta} \cdot \mathrm{diam}(M)^3 &\text{ if } n = 3 \text{ and } b_1(M) = 1 \end{cases}.$$
\end{remark}

\begin{proof}
Let $f \in d_{\mathrm{Whitney}}^{\ast} C^2(M;K)$ satisfy
$$\lambda_{1, \mathrm{Whitney}}^1(M)_{d^*} = \frac{|| d W(f) ||^2_{2,M}}{||W(f)||^2_{2,M}}.$$
There is an orthogonal decomposition in $\Omega^1(M)$
$$W(f) = z + \epsilon,$$
with $\epsilon$   coclosed and $z$   closed.  We will show that $\epsilon$ and $W(f)$ have comparable $L^2$-norms.  Equip the chain group $C_q(M;K)$ with the norm $||\cdot ||_{2,M}$ dual to the $L^2$-Whitney norm on $C^q(M;K)$; see \S \ref{chainnorm} for further discussion.


Because $f\in \mathrm{Im}(d^\ast_{\mathrm{Whitney}}) =  \ker(d)^{\perp_{\mathrm{Whitney}}}=(\text{annihilator of } \mathrm{Im}(\p))^{\perp_{\mathrm{Whitney}}}$, we have 
\begin{align} \label{L2normupperbound}
||W(f)||_{2,M} = \sup_{|| \partial \sigma ||_{2,M} = 1} \left| \int_{\partial \sigma} W(f) \right| 
 = \sup_{||\partial \sigma||_{2,M} = 1} \left| \int_{\partial \sigma} \epsilon \right|.
\end{align}
  The inequality
$$||\cdot||_{2, M} \leq \vol(M)^{1/2} ||\cdot ||_{\infty,M} 
\hspace{0.5cm} \text{on Whitney cochains}$$
implies the dual inequality  
\begin{equation} \label{whitneychainnormcomparison}
\vol(M)^{1/2} || \cdot ||_{2,M}\geq ||\cdot||_{1, M}
 \hspace{0.5cm} \text{on Whitney chains}.
\end{equation}

The inequality \eqref{boundingsupnormbylinftynorm} 
$$c_{M_0}^{-1}||\cdot||_{\infty, M} \leq   ||\cdot ||_{\infty,\text{comb}} \hspace{0.5cm} \text{on Whitney cochains} 
$$
implies the dual inequality  
\begin{equation} \label{whitneychainnormcomparison2}
\  || \cdot ||_{1,\text{comb}}\leq c_{M_0}||\cdot||_{1, M}
 \hspace{0.5cm} \text{on Whitney chains}.
\end{equation}

%
 Suppose $\partial \sigma = \sum a_i c_i$ for cells $c_i$ of the triangulation $K.$  Let $e_0$ denote the maximum length among all 1-cells of $K_.0$ Combining \eqref{L2normupperbound}, \eqref{whitneychainnormcomparison},  and \eqref{whitneychainnormcomparison2} gives
\begin{align} \label{whitneyL2upperbound}
||W(f)||_{2,M} &\leq ||\epsilon||_{\infty,M} \cdot e_0 \cdot  \sup_{||\partial \sigma||_{2,M} = 1} \sum_i |a_i| \nonumber \\
&=||\epsilon||_{\infty,M} \cdot e_0 \cdot \sup_{ \partial \sigma  \not = 0}  \frac{\|\p\sigma\|_{1,\text{comb}}}{ \|\p\sigma\|_{2,M}}\nonumber \\
&\leq c_{M_0} \cdot ||\epsilon||_{\infty,M} \cdot e_0 \cdot  \sup_{ \partial \sigma  \not = 0}  \frac{\|\p\sigma\|_{1,M}}{ \|\p\sigma\|_{2,M}}\nonumber \\
&\leq c_{M_0} \cdot ||\epsilon||_{\infty,M} \cdot e_0 \cdot \vol(M)^{1/2}.
\end{align}

Furthermore, let $|\epsilon|$ achieve its supremum at $p\in M$. 
Then we have for some constant $S_{B_0}$ determined by Garding's inequality for the elliptic operator $d+d^*$ on $B_0 = B_{\frac{1}{2} \mathrm{inj}(M_0)}(p)$ and Sobolev constants for $B_0 \subset \mathbb{H}^n,$
\begin{align} \label{whitneysobolev}
||\epsilon||_{\infty,M} &  \leq S_{B_0} \left( ||  \epsilon||_{2,M} + ||(d + d^\ast)  \epsilon ||_{\infty,M} \right)  \nonumber \\
&= S_{B_0} \left( ||\epsilon||_{2,M} + ||d W(f) ||_{\infty,M} \right) \nonumber \\
&\leq S_{B_0} \left( ||\epsilon||_{2,M} + C_{M_0} ||d W(f) ||_{2,M} \right) \hspace{0.5cm} \text{by Proposition \ref{comparisonswithwhitneysupnorm}} \nonumber \\
&=S_{B_0} \left( ||\epsilon||_{2,M} + C_{M_0} \sqrt{\lambda_{1, \mathrm{Whitney}}^1(M)_{d^\ast}} \cdot ||W(f)||_{2,M} \right).
\end{align}
Inserting  inequality \eqref{whitneysobolev} into \eqref{whitneyL2upperbound} yields
\begin{align} \label{later}
\left(1-G_{M_0} \cdot C_{M_0} \cdot \vol(M)^{1/2} \sqrt{\lambda_{1, \mathrm{Whitney}}^1(M)}  \right) ||W(f)||_{2,M} &\leq G_{M_0} \cdot \vol(M)^{1/2}||\epsilon||_{2,M},
\end{align}

where $G_{M_0} := c_{M_0} \cdot e_0 \cdot S_{B_0}.$  If $\lambda_{1,\mathrm{Whitney}}^1(M)_{d^\ast} \leq  \frac{1}{4 G_{M_0}^2 C_{M_0}^2\vol(M)},$ then  
\begin{equation*}
||W(f)||_{2,M}^2 \leq 4G_{M_0}^2\vol(M)  \cdot ||\epsilon||_{2,M}^2.
\end{equation*}
Therefore,
\begin{align*}
\lambda_1^1(M)_{d^\ast} &\leq \frac{||d\epsilon||_{2,M}^2}{||\epsilon||_{2,M}^2} \\
&= \frac{||dW(f)||_{2,M}^2}{||\epsilon||_{2,M}^2} \\
&\leq  4G_{M_0}^2\vol(M) \cdot \frac{||dW(f)||_{2,M}^2}{||W(f)||_{2,M}^2} \\
&= 4G_{M_0}^2\vol(M) \cdot \lambda_{1,\mathrm{Whitney}}^1(M)_{d^\ast}.
\end{align*}
\end{proof}

\section{Applications} \label{applications}
\subsection{Naive lower bounds on $\lambda_1^1(M)$}
\begin{proposition} \label{exponentialspectralgap}
Let $M_0$ be a closed hyperbolic $n$-manifold.  Let $M \rightarrow M_0$ be an arbitrary finite cover with $b_1(M) = 0.$  Then
$$\frac{1}{\lambda_1^1(M)} \leq \exp( H_{M_0} \vol(M))$$
for some constant $H_{M_0}$ depending only on $M.$
\end{proposition} 

\begin{remark}
By the Cheeger-M\"{u}ller theorem, under the assumption $b_1(M) = 0,$  
$$\limsup_{M} \frac{\log \frac{1}{\lambda_1^1(M)}}{\vol(M)} \leq \frac{1}{6\pi},$$
as $M$ varies through any sequence of closed hyperbolic 3-manifold Benjamini-Schramm converging to $\mathbb{H}^3.$  However, we are unaware of any upper bound for $\frac{1}{\lambda_1^1(M)}$ for higher dimensional hyperbolic manifolds in the literature. 
\end{remark}

\begin{proof}
By Lemma \ref{lemmaalternatelowerboundlambda10}, 
$$\frac{1}{\lambda_1^1(M)_d} \leq C \cdot \mathrm{diam}(M)^2 \cdot \vol(M)$$
for some constant $C$ depending only on a lower bound for the injectivity radius of $M.$  So, we focus our attention on $\lambda_1^1(M)_{d^\ast}.$ 

Let $K_0$ be a triangulation of $M_0.$  Let $K$ be the pullback triangulation of $M.$  Consider the operator 
$$A := \partial_1^\ast \partial_2:C_2(K) \rightarrow C_2(K).$$
$A$ is a sparse-integer matrix, i.e. every column has a bounded number of entries (upper bound depending only on $M_0$).  By Hadamard's inequality, every $k \times k$-minor has determinant of absolute value at most $\exp(O_{M_0}(k)).$  Let $N = \dim C_2(K) \approx_{M_0} \vol(M).$  If the characteristic polynomial of $A$ equals $x^N + a_{N-1} x^{N-1} + \cdots + a_{k+1} x^{k+1} + a_k x^k$ (where $a_k$ is the last non-zero coefficient), then 
\begin{align*}
\sum_{\lambda = \text{non-zero e.value of } A} \frac{1}{\lambda} &= \frac{|a_{k+1}|}{|a_k|} \leq |a_{k+1}|
\end{align*}
because $|a_k|$ is an integer $\geq 1.$

But $a_{k+1}$ is the sum of the $\binom{N}{k+1} \leq 2^N$ principal $(k + 1) \times (k+1)$ principal minors of $A,$ all of which have absolute value at most 
$\exp(O_{M_0}(k))$ by our earlier remark.  Therefore,
$$\frac{1}{\lambda_{1, \mathrm{comb}}^1(M)_{d^\ast}} \leq \sum_{\lambda = \text{non-zero e.value of } A} \frac{1}{\lambda} \leq \exp(O_{M_0} \vol(M)).$$
By (the proof of) Theorems \ref{whitneylambda1controlsscl}, \ref{whitneylambda1controlsscltriangulationindependent}, and Remark \ref{combinatoriallambda1controlsscl}, there is an upper bound
$$\frac{\mathrm{scl}(\gamma)}{\ell(\gamma)} \ll_{M_0} \frac{1}{\lambda_{1, \mathrm{comb}}^1(M)_{d^\ast}} = \exp(O_{M_0} \vol(M)).$$

In particular, by Corollary \ref{lowerboundeigenvalue} and the diameter bounds from Propositions \ref{diametercomparison} and \ref{upperbounddiameterdirichletdomain}, there is an upper bound 
\begin{equation*}
\frac{1}{\lambda_1^1(M)_{d^\ast}} = \exp(O_{M_0} \vol(M)).
\end{equation*}
\end{proof}

\subsection{Improved lower bounds on $\lambda_1^1(M^n)$ for hyperbolic $n$-manifolds, $n > 3$} \label{1formshighdimension}
\begin{proposition} \label{smallsurfaceshighdimension}
Let $M_0$ be a closed hyperbolic $n$-manifold, $n > 3.$  Fix a constant $C > 0.$  Suppose $M \rightarrow M_0$ is an arbitrary finite cover satisfying $\lambda_1^1(M) \gg \vol(M)^{-C}.$ 
Suppose some multiple of $\gamma \in \pi_1(M)$ bounds.  Then 
$$\frac{\mathrm{scl}(\gamma)}{\ell(\gamma)} \ll_{M_0} \vol(M)^{1 + \frac{C}{2}} \cdot \mathrm{diam}(M).$$
\end{proposition}

\begin{proof}
This follows immediately from Theorem \ref{whitneylambda1controlsscltriangulationindependent} and  Proposition \ref{whitneyderhamcomparison}.  
\end{proof}

\begin{remark} \label{1formspectrahyperbolicmanifolds}
The bottom of the $1$-form spectrum $\lambda_1^1(\mathbb{H}^n)$ for the Laplacian $\Delta_1$ acting on smooth compactly supported 1-forms on $\mathbb{H}^n$ equals $\left(\frac{n-3}{2} \right)^2$ for $n \geq 3$ \cite[Theorem 1]{Donnelly}.

In particular, 1-form eigenvalues less than $\lambda_1^1(\mathbb{H}^n)$ are \emph{exceptional}.  There are natural families of closed hyperbolic $n$-manifolds $M, n > 3,$ such as arithmetic congruence hyperbolic $n$-manifolds, for which for which $\lambda_1^1(M)$ is uniformly bounded below \cite{BC}.  For such families, we may set $C = 0$ in Proposition \ref{smallsurfaceshighdimension}.  More generally, it seems plausible to us that if $M_0$ is a closed hyperbolic $n$-manifold, $n > 3,$ and $M \rightarrow M_0$ is an aribtrary finite cover, then $\lambda_1^1(M) \gg_{M_0} \vol(M)^{-C},$ for some constant $C.$ 

\end{remark}

\subsection{Lower bounds on $\lambda_1^1(M^3)$ using retractions from hyperbolic $n$-manifolds, $n > 3$}
\begin{proposition} \label{stealingboundsfromhighdimension}
Let $N_0$ be a closed hyperbolic $n$-manifold, $n > 3.$  Let $M_0 \subset N_0$ be a totally geodesic submanifold.  Suppose $N \xrightarrow{\pi} N_0$ is an arbitrary finite cover. 
Let $M =$ (a connected component of) $\pi^{-1}(M_0).$  Suppose that there is a covering $p: N' \rightarrow N$ of degree $d$ for which
\begin{itemize}
\item
the submanifold $M$ lifts to $N'$

\item
$N'$ retracts onto $M.$
\end{itemize}
Suppose some integer multiple of $\gamma \in \pi_1(M)$ bounds.  Then
$$\frac{\mathrm{scl}(\gamma)}{\ell(\gamma)} \ll_{N_0} d^2 \cdot \vol(N) \cdot \mathrm{diam}(N) \cdot \sqrt{\frac{1}{\lambda_1^1(N')_{d^\ast}}}.$$
\end{proposition}

\begin{remark}
The work of Bergeron-Haglund-Wise \cite{BHW} produces many interesting examples satisfying the hypotheses of Proposition \ref{stealingboundsfromhighdimension}. 

The main theorems of the present paper relate $\mathrm{scl}$ and $\frac{1}{\lambda_1^1}.$  Proposition \ref{stealingboundsfromhighdimension} punts the difficulty of bounding $\lambda_1^1(M)$ below to that of bounding $\lambda_1^1(N')$ below.  \emph{This should be regarded as a significant gain, since the 1-form spectrum of $N'$ should be much easier to bound away from 0 than the 1-form spectrum of $M$}; see Remark \ref{1formspectrahyperbolicmanifolds}.  In particular, modulo the hope expressed in Remark \ref{1formspectrahyperbolicmanifolds} 
and assuming that $d$ can be taken polynomial in $\vol(N),$ 
Proposition \ref{stealingboundsfromhighdimension} will produce a rich family of examples of $M$ for which $\lambda_1^1(M) \gg \vol(M)^{-C}$ for some constant $C.$ 
\end{remark}

\begin{proof}
Let $\gamma \in \pi_1(M) \subset \pi_1(N)$ be as in the proposition statement.  Let $N'$ be the covering realizing the retraction onto $M.$  By Theorem \ref{whitneylambda1controlsscltriangulationindependent} and Proposition \ref{whitneyderhamcomparison},
\begin{align*} \label{lambda1retract}
\frac{\mathrm{scl}_{N'}(\gamma)}{\ell_{N'}(\gamma)} &\ll_{N_0} \vol(N')^{1/2} \cdot \mathrm{diam}(N') \cdot \sqrt{\frac{1}{\lambda_{1,\mathrm{Whitney}}^1(N')_{d^\ast}}} \nonumber \\
&\ll_{N_0} \vol(N')^{1/2} \cdot \mathrm{diam}(N') \cdot \vol(N')^{1/2} \cdot \sqrt{\frac{1}{\lambda_1^1(N')_{d^\ast}}} \nonumber \\
&\ll_{N_0} d^2 \cdot \vol(N) \cdot \mathrm{diam}(N) \cdot \sqrt{\frac{1}{\lambda_1^1(N')_{d^\ast}}}.
\end{align*}
Also,
$$\mathrm{scl}_M(\gamma) \leq \mathrm{scl}_{N'}(\gamma) \text{ and } \ell_M(\gamma) = \ell_{N'}(\gamma),$$
the latter because $M$ is geodesically embedded in $N'$ and the former because the retraction $p_{\ast} : \pi_1(N') \rightarrow \pi_1(M)$ reduces commutator length.  The conclusion follows.
\end{proof}

\begin{remark}
We emphasize that Proposition \ref{stealingboundsfromhighdimension} does not require any cohomology vanishing hypothesis.  The two key inputs for Proposition \ref{stealingboundsfromhighdimension} are Propositions \ref{whitneylambda1controlsscl} and \ref{whitneyderhamcomparison}.  And indeed, 

\begin{itemize}
\item
Proposition \ref{whitneylambda1controlsscl} upper bounds $\mathrm{scl}_{N'}(\gamma)$ in terms of $\frac{1}{\lambda_1^1(N)_{d^\ast}}$ provided some multiple of $\gamma \in \pi_1(N')$ bounds; no supplementary cohomology vanishing hypothesis is required.   

\item
Proposition \ref{whitneyderhamcomparison} proves $\frac{1}{\lambda_{1,\mathrm{Whitney}}^1(N')_{d^\ast}} \ll_{N_0} \frac{1}{\lambda_{1,\mathrm{Whitney}}^1(N')_{d^\ast}}$; no supplementary cohomology vanishing hypothesis is required.  
\end{itemize}
\end{remark}

\begin{corollary}
Same notation and hypotheses as Proposition \ref{stealingboundsfromhighdimension}.  Suppose in addition that $b_1(M) = 0.$  Then
$$\frac{1}{\lambda_1^1(M)} \ll_{M_0} \mathrm{diam}(M)^2 \cdot \left( d^2 \cdot \vol(N) \cdot \mathrm{diam}(N) \right)^2 \cdot \frac{1}{\lambda_1^1(N')}.$$
\end{corollary}

\begin{proof}
This follows directly from Proposition \ref{stealingboundsfromhighdimension} and Corollary \ref{lowerboundeigenvalue}. 
\end{proof}

\appendix
\section{Estimating $\lambda_1^0(M)$} \label{alternatelowerboundlambda10} 
In this section we give a weak lower bound for the first nonzero eigenvalue of the Laplacian acting on functions on a hyperbolic $n-$manifold. With more work, the bound can be considerably improved, but the easy given bound suffices for our purposes. 
\begin{lemma} \label{lemmaalternatelowerboundlambda10}
There exists $C>0,$ depending only on the minimum of $1$ and the injectivity radius of $M,$ so that 
\begin{align}\lambda_1^0\geq \frac{C}{\mathrm{diam}(M)^2 \cdot \vol(M)}.
\end{align}
\end{lemma}
\begin{proof}
Let $u\in C^\infty(M)$ with $\|u\|_{L^2} =1$ and $\Delta u =\lambda_1^0(M)u$.  
Then $\|du\|_{L^2}^2 = \lambda_1^0(M).$ By Proposition \ref{supnormeigenfunction}, 
\begin{align}\label{gradnorm}\|du\|_{L^\infty}\leq \sqrt{\lambda_1^0} \cdot C \left(n,1,\frac{\mathrm{inj}(M)}{2},\lambda \right).
\end{align}
 Since $u\perp_{L^2}1$ and has $L^2$ norm one, there exist $p_1,p_2\in M$ so that 
$u(p_1) = 0$ and $u(p_2) = \frac{1}{\sqrt{\vol(M)}}$. Then 
\begin{align}\label{suppnorm} \frac{1}{\sqrt{\vol(M)}} = |u(p_2)| \leq d(p_1,p_2)\cdot \sqrt{\lambda} \cdot C \left(n,1,\frac{\mathrm{inj}(M)}{2},\lambda \right).
\end{align}
Hence 
\begin{align}\label{lambda1bnd} \frac{C \left(n,1,\frac{\mathrm{inj}(M)}{2},\lambda \right)^{-2}}{\mathrm{diam}(M)^2 \cdot \vol(M) }   \leq  \lambda .
\end{align}
Since $C(n,1,L,\lambda)$ is a decreasing function of $L$, the result follows. 
\end{proof}

\end{document}